\documentclass[12pt]{amsart}
\pdfoutput=1
\usepackage{amsmath, amssymb, amsthm, bm}
\usepackage[colorlinks=true, linkcolor=blue, urlcolor=blue, citecolor=blue]%
  {hyperref}
\usepackage[noabbrev, capitalize]{cleveref}
\usepackage{mathtools}
\usepackage{rsfso}
\usepackage{paralist}
\usepackage[dvipsnames]{xcolor}
\usepackage{tikz}
\usetikzlibrary{arrows, shapes, positioning, calc, patterns, cd}
\usepackage[alphabetic, backrefs]{amsrefs}
\usepackage[T1]{fontenc}
\usepackage{mathptmx}
\usepackage{microtype}
\usepackage[format=hang, font=small, justification=raggedright,
  labelfont=bf]{caption}
\usepackage[centering, includeheadfoot, hmargin=1.0in, tmargin=1.0in,
  bmargin=1in, headheight=7pt]{geometry}
\newtheorem{lemma}{Lemma}[section]
\newtheorem{theorem}[lemma]{Theorem}
\newtheorem{corollary}[lemma]{Corollary}
\newtheorem{proposition}[lemma]{Proposition}
\theoremstyle{definition}
\newtheorem{definition}[lemma]{Definition}
\newtheorem{remark}[lemma]{Remark}
\newtheorem{examplex}[lemma]{Example}
\newenvironment{example}
  {\pushQED{\qed}\examplex}
  {\popQED\endexamplex}

\AddToHook{env/theorem/begin}{\crefalias{lemma}{theorem}}
\AddToHook{env/corollary/begin}{\crefalias{lemma}{corollary}}
\AddToHook{env/proposition/begin}{\crefalias{lemma}{proposition}}
\AddToHook{env/definition/begin}{\crefalias{lemma}{definition}}
\AddToHook{env/remark/begin}{\crefalias{lemma}{remark}}
\AddToHook{env/example/begin}{\crefalias{lemma}{example}}

\makeatletter

  \let\c@lemma\c@figure
\makeatother

\newcommand{\relphantom}[1]{\mathrel{\phantom{#1}}}

\newcommand{\ceil}[1]{\ensuremath{\left\lceil #1 \right\rceil}}
\newcommand{\floor}[1]{\ensuremath{\left\lfloor #1 \right\rfloor}}
\newcommand{\ideal}[1]{\ensuremath{\left\langle #1 \right\rangle}}
\newcommand{\set}[1]{\ensuremath{\left\{ #1 \right\} }}
\newcommand{\colequal}{\ensuremath{\mathrel{\mathop :}=}}
\newcommand{\RL}{\ensuremath{L_{\RR}}}
\newcommand{\kk}{\ensuremath{\Bbbk}}
\newcommand{\NN}{\ensuremath{\mathbb{N}}}
\newcommand{\PP}{\ensuremath{\mathbb{P}}}

\newcommand{\RR}{\ensuremath{\mathbb{R}}}
\newcommand{\ZZ}{\ensuremath{\mathbb{Z}}}
\newcommand{\cHom}{\ensuremath{\!\mathcal{H}\!\!\textit{om}}}

\DeclareMathOperator{\id}{id}
\DeclareMathOperator{\Hom}{Hom}
\DeclareMathOperator{\Ker}{Ker}
\DeclareMathOperator{\Pic}{Pic}

\begin{document}

\vspace*{-3.2em}

\title[Cellular free resolutions]%
  {Cellular free resolutions for normalizations of toric ideals}

\author[C.~Berkesch]{Christine Berkesch}
\address{Christine Berkesch: School of Mathematics, University of
  Minnesota, Minneapolis, Minnesota, 55455, United States of America;
  {\normalfont \texttt{cberkesc@umn.edu}}}

\author[L.~Cranton Heller]{Lauren Cranton Heller}
\address{Lauren Cranton Heller: Department of Mathematics, University of
  Nebraska, Lincoln, Nebraska, 68588, United
  States of America; {\normalfont \texttt{lheller2@unl.edu}}}

\author[G.G.~Smith]{Gregory G.{} Smith}
\address{Gregory G.{} Smith: Department of Mathematics \& Statistics, Queen's
  University, Kingston, Ontario, K7L 3N6, Canada;
  {\normalfont \texttt{ggsmith@mast.queensu.ca}}}

\author[J.~Yang]{Jay Yang}
\address{Jay Yang: Department of Mathematics, Vanderbilt University,
  Nashville, Tennessee, 37240, United States of America;
  {\normalfont \texttt{jay.k.yang@vanderbilt.edu}}}

\subjclass[2020]{13D02; 14M25, 05E14}


\begin{abstract}
  For any toric ideal $I$ in a polynomial ring $S$, we provide a combinatorial
  description of a free resolution of the integral closure of the $S$-module
  $S/I$.  These new complexes arise from an extension of Bayer--Sturmfels'
  theory of cellular free resolutions.  As applications, we unify several
  constructions for a resolution of the diagonal embedding of a toric variety,
  and compare the locally free resolutions for toric subvarieties introduced
  by Hanlon--Hicks--Lazarev and Brown--Erman.
\end{abstract}

\maketitle

\vspace*{-1.5em}
\section{Overview}
\label{s:intro}

\noindent
Explicit free resolutions not only hold a foundational place in commutative
algebra~\cites{MS05, Pee11}, but also play an outsized role in algebraic
combinatorics~\cite{OW07}, algebraic geometry~\cite{Eis05}, and representation
theory~\cite{CG10}. The Koszul complex, which resolves ideals generated by
regular sequences, is the archetypal example. Extending this paradigm, the
Eagon--Northcott, Buchsbaum--Rim, and Gulliksen--Neg\r{a}rd complexes resolve
certain ideals generated by minors of matrices; see~\cite{Wey03}. The theory
of cellular free resolutions~\cites{BPS98, BS98}, generalizing the Taylor and
Eliahou--Kervaire resolutions, offers the most comprehensive approach for
constructing combinatorial free resolutions of monomial ideals by linking them
to cell complexes. While \cites{BS98, PS98} formulate an analog for binomial
ideals, the theory is much less well behaved or fully developed. The
primary goal of this article is to significantly expand the framework for
cellular free resolutions of binomial ideals. In particular, our construction
yields a cellular free resolution for the integral closure of the quotient of
a polynomial ring by a toric ideal.

To elaborate, fix a positive integer $n$ and let $S \colequal \kk[x_1, x_2,
  \dotsc, x_n]$, where $\kk$ is a field.  A saturated lattice $L \subseteq
\ZZ^{n}$ is a free abelian subgroup containing all integral points in its
rational span.  Such saturated lattices are in bijection with toric ideals in
the polynomial ring $S$:
\begin{align*}
  L
  &\subseteq \ZZ^{n}
  &&\longleftrightarrow
  & I_{L}
  &\colequal \ideal{\bm{x}^{\bm{u}} - \bm{x}^{\bm{v}} \in S
    \mathrel{\big|} \text{$\bm{u} - \bm{v} \in L$
    for all $\bm{u}, \bm{v} \in \NN^{n}$}} \, .
\end{align*}
Endow the real vector space $\RL \colequal L \otimes_{\ZZ} \RR$ with the
structure of an $L$\nobreakdash-equivariant cell complex. A \emph{compatible
  $\ZZ^n$-stratification} is a map $\psi \colon \RL \to \ZZ^{n}$ that is
constant on the relative interior of each cell in $\RL$ and compatible with
translations by elements of $L$; see \Cref{d:strat}. Homogenizing the cellular
$\kk$-complex of $\RL$ using $\psi$ yields a cellular free $S[L]$-complex
$F_{\psi}$ whose $i$th term is
\[
  (F_{\psi})_{i}
  = \bigoplus_{\begin{subarray}{c}
      \sigma \subset \RL \\
      \dim \sigma = i
    \end{subarray}}
  S \bigl(- \psi(\sigma) \kern-1.0pt \bigr)
\]
and whose differentials are detailed in \Cref{d:comp}.

The algebraic counterpart of passing from the universal cover $\RL$ to the
real torus $\RL \mathbin{\!/\!} L$ is realized via the extension-of-scalars
functor $M \mapsto M \otimes_{S[L]} S$, which transforms any $\ZZ^{n}$-graded
$S[L]$-module into a $(\ZZ^n \mathbin{\!/\!} L)$-graded $S$-module; see
\Cref{r:com}.  Our main contribution is the following theorem.

\begin{theorem}
  \label{t:main}
  When $\psi \colon \RL \to \ZZ^{n}$ is defined by
  $\psi(\bm{p}) \colequal \ceil{\bm{p}}$ for all
  $\bm{p} \in \RL \subseteq \RR^{n}$, the cellular free $S$-complex
  $F_{\psi} \otimes_{S[L]} S$ is a $(\ZZ^n \mathbin{\!/\!} L)$-graded
  resolution of the integral closure of the $S$-module
  $S \mathbin{\!/\!} I_{L}$.
\end{theorem}

\noindent
The integral closure of the domain $S/I_L$ is taken relative to its field of
fractions, an operation also called normalization.

A uniform combinatorial description of a free resolution of the integral
closure of $S \mathbin{\!/\!} I_{L}$ for an arbitrary toric ideal $I_{L}$ is
new. The key innovation is the introduction of compatible
$\ZZ^{n}$-stratifications, which allow nonintegral vertices (or $0$-cells) in
$\RL$ to be labeled by monomials in $S$. When the vertices lie in the lattice
$L$, this framework recovers the constructions in \cite{BS98}*{\S3} and
\cite{MS05}*{\S9.2}. In particular, the minimal free resolution of a
unimodular Lawrence ideal described in \cite{BPS01}*{Corollary~3.6} arises as
a special case.

By decoupling vertices from lattice points, we expand the range of viable cell
structures on the space $\RL$. Both \Cref{e:And-ep} and \Cref{l:FH} illustrate
this flexibility. For any compatible $\ZZ^{n}$-stratification, \Cref{p:res},
\Cref{c:res}, and \Cref{l:cond} develop the fundamental theory of cellular
free $S$-complexes. Furthermore, \Cref{t:main} implies that the integral
closure of $S \mathbin{\!/\!} I_{L}$ is Cohen--Macaulay, thereby providing an
alternative proof of \cite{Hoc72}*{Theorem~1}.  Taken together, these advances
substantially enhance the power and scope of the theory of cellular free
resolutions.

\subsection*{Primary geometric application}
Our original motivation was to recast the locally free resolutions of
\cite{HHL24} in the language of combinatorial commutative algebra. To this
end, consider a torus-equivariant closed embedding $\varphi \colon Y \to X$ of
a normal toric variety $Y$ into a smooth toric variety $X$ with no torus
factors. In this setting, the polynomial ring $S$ is the Cox ring of $X$, the
image $\varphi(Y)$ is cut out by a toric ideal in $X$, and $L \subseteq
\ZZ^{n}$ is the corresponding lattice; see \Cref{r:toric}. Let $\psi
\colon \RL \to \ZZ^{n}$ be the ceiling stratification; see \Cref{d:ceil}.

The main result~\cite{HHL24}*{Theorem~A} constructs an explicit resolution by
line bundles of the $\mathcal{O}_{X}$-module $\varphi_{*}^{}
\mathcal{O}_{Y}$. For the ceiling stratification, \Cref{t:hhl} identifies this
resolution with the $\mathcal{O}_{X}$-complex associated to the cellular free
$S$-complex $F_{\psi} \otimes_{S[L]} S$. This perspective circumvents the
inductive arguments and discrete Morse theory used in
\cite{HHL24}*{\S1.3}.

The preprint \cite{BH25}, which acknowledges that it was
``inspired by an Oberwolfach talk by [the first author],'' employs essentially
the same cellular techniques; see \cite{Ber25}. They address only the ceiling
stratification, as in \Cref{s:ceil}.  When the lattice $L$ is saturated their
main result, \cite{BH25}*{Theorem~2.9}, is a special case of \Cref{t:main};
also see \Cref{r:nonSatDetails}.

Beyond this reinterpretation, our algebraic approach demonstrates that the
locally free resolution corresponds to an actual resolution of $S$-modules,
rather than merely a virtual resolution in the sense of
\cite{BES20}*{Definition~1.1}. Although \cite{HHL24}*{\S2.3} already
emphasizes the ceiling stratification, \Cref{p:nor} and \Cref{t:resSat}
further reveal its strong connection to integral closures. In this way, the
cellular free $S$-complexes provide a remarkably simple and concise source for
all of these locally free resolutions.

\subsection*{Understanding virtual resolutions}

More significantly, the cellular free complexes arising from a ceiling
stratification provide a framework for comparing the locally free resolutions
in \cite{HHL24}*{Theorem~3.5} and \cite{BE24}*{Theorem~1.2}. Specifically,
\Cref{c:hhl} shows that the resolution in \cite{BE24}*{Theorem~1.2} is
isomorphic to a direct summand of that in \cite{HHL24}*{Theorem~3.5}.

For the diagonal embedding $\Delta \colon Y \to X \colequal Y \times Y$ of a
smooth toric variety $Y$ with no torus factors, \Cref{t:cell-min} and
\Cref{c:HHL-BE} identify the resolutions by line bundles in both
\cite{BE24}*{Theorem~1.2} and \cite{HHL24}*{Theorem~3.5} with the
$\mathcal{O}_{X}$-complex associated to the minimal cellular free resolution
of the integral closure of $S \mathbin{\!/\!} I_L$. This corollary justifies
the belief expressed in \cite{BE24}*{p.3} and provides an explicit description
of the terms; see also \cite{BE24}*{Remark~1.6}. In essence, these cellular
free $S$-complexes completely explain the relationship between the virtual
resolutions in \cite{BE24}*{Theorem~1.2} and \cite{HHL24}*{Theorem~3.5}.

\subsection*{Further geometric applications}

By exploiting compatible $\ZZ^n$-stratifications other than the ceiling
stratification, we incorporate additional resolutions of the diagonal for the
toric variety $Y$ within our framework.  For the construction in
\cite{And24}*{Definition~3.4}, this is accomplished by introducing an appropriate
compatible $\ZZ^{n}$-stratification; see \Cref{d:And}. \Cref{e:And}
underscores some of the subtleties connected with the associated cellular free
complexes.  For the resolution of the diagonal in \cite{FH25}*{Example~3.15},
the primary challenge is identifying the relevant cell structure, and
\Cref{p:FH} clarifies the case $Y = \PP^{2}$.

\subsection*{Conventions}

Throughout the document, $\NN$ denotes the set of nonnegative integers, and
$\kk$ denotes a field. The phrase `cell complex' is synonymous with `CW
complex.'

\subsection*{Outline}

Section~\ref{s:cell} introduces compatible $\ZZ^n$-stratifications and extends
the theory of cellular free resolutions to this broader
setting. Section~\ref{s:ceil} shows that the componentwise ceiling function is
a compatible $\ZZ^n$-stratification which yields a free resolution of the
integral closure of the comodule of a toric ideal. Section~\ref{s:app}
explores applications of this ceiling stratification, demonstrating that our
construction recovers the virtual resolutions in \cite{HHL24}*{Theorem~3.5}
and elucidates the locally free resolutions in
\cite{BE24}*{Theorem~1.2}. Section~\ref{s:oth} situates other resolutions of
the diagonal---namely those by \cite{And24}*{Theorem~1.1} and
\cite{FH25}*{Example~3.15}---within this unified framework.

\section{Cellular free resolutions from stratifications}
\label{s:cell}

\noindent
In this section, we construct free resolutions of monomial modules by
extending the cellular techniques from \cite{BS98}*{\S3} and
\cite{MS05}*{\S4}. Chiefly, we allow for more general monomial labelings of
the cells by introducing a stratification. We also remove the regularity
assumption on the underlying cell complex and the positivity assumption on the
lattice $L$, allowing $L \cap \NN^{n} \neq \set{\bm{0}}$.

\subsection*{Topological context}

Fix a positive integer $n$.  The \emph{lattice} $L$ is a subgroup of the free
abelian group $\ZZ^n$.  Let $d$ be the rank of $L$, and let $\RL \colequal L
\otimes_{\ZZ} \RR \cong \RR^{d}$ be the associated real vector space endowed
with the Euclidean topology. The quotient space $\RL \mathbin{\!/\!} L$ is a
real topological torus. The inclusion $L \subseteq \ZZ^n$ gives rise to an
inclusion of real vector spaces $\RL \subseteq \RR^n$. We tacitly regard the
elements of $L$ or $\RL$ as elements of $\ZZ^n$ or $\RR^n$ respectively.

A cell complex in $\RL$ is \emph{$L$\nobreakdash-equivariant} if, for each open
cell $\sigma$ and each lattice point $\bm{v} \in L$, the translate
$\sigma + \bm{v}$ is also an open cell.  Assume that the topological space
$\RL$ is endowed with the structure of an $L$-equivariant cell complex; see
\cite{Hat02}*{Appendix} for more on cell complexes.  Via the quotient map
$\pi \colon \RL \to \RL \mathbin{\!/\!} L$, the $L$-equivariant cell complex
on $\RL$ equips the real topological torus $\RL\mathbin{\!/\!} L$ with the
structure of a cell complex.

The following is our favorite source for $L$-equivariant cell complex
structures on the space $\RL$.

\begin{remark}[Cell complex arising from a periodic arrangement]
  \label{r:cell}
  As in \cite{DP11}*{\S2.1}, a central hyperplane arrangement in the space
  $\RL$ is the set of hyperplanes
  $\set{\bm{x} \in \RL \mathrel{\big|} \bm{u}_i \cdot \bm{x} = 0}$ for some
  finite set of vectors $\bm{u}_{1}, \bm{u}_{2}, \dotsc, \bm{u}_{m}$ in
  $\Hom(L, \ZZ)$. The \emph{periodic arrangement} $\mathcal{A}$ consists of
  all affine hyperplanes of the form
  $H \colequal \set{ \bm{p} \in \RL \mathrel{\big|} \bm{u}_{i} \cdot \bm{p} =
    j}$, where $1 \leqslant i \leqslant m$ and $j$ is an integer.  A
  \emph{chamber} $\sigma$ of the arrangement $\mathcal{A}$ is a connected
  component of the complement $\RL \setminus \bigcup_{H \in \mathcal{A}} H$.
  Hence, the closure $\overline{\sigma}$ is a convex polyhedron. A \emph{face}
  of $\mathcal{A}$ is the face of the polyhedron $\overline{\sigma}$: the
  intersection of $\overline{\sigma}$ with a supporting hyperplane.  The faces
  of $\mathcal{A}$ endow the space $\RL$ with the structure of a polyhedral
  cell complex; see \cite{MS05}*{Definition~4.1}. Unlike in
  \cite{BPS01}*{\S3}, the vertices (or $0$-cells) of this polyhedral cell
  complex need not be lattice points.  Under the quotient map
  $\pi \colon \RL \to \RL \mathbin{\!/\!} L$, this periodic arrangement
  defines a \emph{toric arrangement} in $\RL \mathbin{\!/\!} L$; see
  \cite{DP11}*{Definition~14.2}.
\end{remark}

A specific instance of this construction is significant enough to warrant a
name.

\begin{definition}
  \label{d:std}
  The \emph{standard cell structure} on the space $\RL \subseteq \RR^n$ is the
  polyhedral cell structure determined by the periodic arrangement arising
  from the standard basis vectors in $\RR^n$.
\end{definition}

We exploit the standard cell structure in toric geometry. For clarity and
brevity, we focus on smooth toric varieties, with extensions to singular toric
varieties and toric stacks left to the reader.

\begin{remark}[Cell complex arising from a toric embedding]
  \label{r:toric}
  Fix a smooth toric variety $X$ with no torus factors; see
  \cite{CLS11}*{Corollary~3.3.10}. Let $\varphi \colon Y \to X$ be a
  torus-equivariant embedding of a normal toric variety $Y$.  The toric
  morphism $\varphi$ corresponds to an injective $\ZZ$-linear map
  $\overline{\varphi} \colon N_{Y} \to N_{X}$ between the lattices of
  one-parameter subgroups of $Y$ and $X$; see \cite{CLS11}*{Theorem~3.3.4}.
  Dualizing the map $\overline{\varphi}$ gives the surjective $\ZZ$-linear map
  $\overline{\varphi}^{*} \colon M_{X} \to M_{Y}$ between the character
  lattices of $X$ and $Y$. Set $m$ and $n$ to be the number of rays in the
  fans of $Y$ and $X$, respectively. The groups of torus-invariant Weil
  divisors on $Y$ and $X$ are identified with $\ZZ^m$ and $\ZZ^n$,
  respectively. Thus, we obtain the commutative diagram
  \[
    \begin{tikzcd}[row sep=0.8em, column sep=3.0em]
      & 0 \arrow[d] &&& \\
      & L \arrow[dd, "\kappa" left]
      \arrow[ddr, "\iota" above] &&& \\ \\
      0 \arrow[r] & M_X \arrow[dd, "\overline{\varphi}^{*}" left]
      \arrow[r, "\nu" below]
      & \ZZ^n \arrow[r] \arrow[dd] & \Pic(X) \arrow[r] \arrow[dd]
      & 0 \, \phantom{.} \\ \\
      0 \arrow[r] & M_Y \arrow[r] \arrow[d] & \ZZ^m \arrow[r]
      & \Pic(Y) \arrow[r] & 0 \, , \\
      & 0 &&&
    \end{tikzcd}
  \]
  where $L \colequal \Ker(\overline{\varphi}^{*})$, the injection $\kappa$ is
  the canonical inclusion, and the rows in the matrix of the map $\nu$ are the
  primitive lattice points generating the rays in the fan of $X$; see
  \cite{CLS11}*{Theorem~4.2.1}.  The image of the composite map
  $\iota \colequal \nu \circ \kappa$ realizes the lattice $L$ as a subgroup of
  $\ZZ^{n}$ with rank $d \colequal \dim X - \dim Y$.  In this case, we always
  equip the space $\RL$ with its standard cell structure.
\end{remark}

We illustrate this toric construction with a concrete example.

\begin{example}[Cell complex arising from the inclusion of a point into
  Hirzebruch surface]
  \label{e:H2}
  Let
  $X \colequal \PP \bigl( \mathcal{O}_{\PP^1} \kern-0.5pt \oplus
  \mathcal{O}_{\PP^1}(-2) \kern-1.0pt \bigr)$ denote the second Hirzebruch
  surface. The rays in the fan of the toric variety $X$ are generated by the
  primitive lattice points $(1, 0)$, $(0, 1)$, $(-1, 2)$, and $(0, -1)$; see
  \cite{CLS11}*{Example~3.1.16}.  Consider the inclusion of the identity point
  $Y \colequal \set{[1 \mathbin{:} 1 \mathbin{:} 1 \mathbin{:} 1 ]}$ in the
  dense torus of $X$.  This corresponds to the unique $\ZZ$-linear map
  $\overline{\varphi} \colon \ZZ^{0} \to \ZZ^{2}$.  The lattice
  $L$ has rank~$2$, and the matrix of the inclusion
  $\iota \colon \ZZ^2 \cong L \to \ZZ^4$, with respect to the standard bases, is
  \[
    \begin{bmatrix*}[r]
      1 & 0 & -1 &  0 \\[-2pt]
      0 & 1 &  2 & -1 \\
    \end{bmatrix*}^{\textsf{T}} \, .
  \]
  The standard cell structure on the real vector space $\RL \cong \RR^2$ is
  induced by the standard basis in $\RR^4$. Identifying $\RL$ with the real vector
  space $\RR^2$, the polyhedral cell structure arises from the rows of this matrix
  under the standard basis.  Hence, a fundamental domain for the
  $L$-action on $\RL$ contains two $0$-cells, five $1$-cells, and three
  $2$-cells, as depicted in \Cref{f:H2lab}.

  As the hyperplanes with normal vectors $(0,1)$ and $(0,-1)$ coincide, only
  three colors are needed to display the fundamental domain.  The hyperplanes
  with normal vectors $(1,0)$ and $(-1,2)$ intersect at a point outside of
  $L$, so each gives rise to two distinct $1$-cells in the fundamental domain.
\end{example}

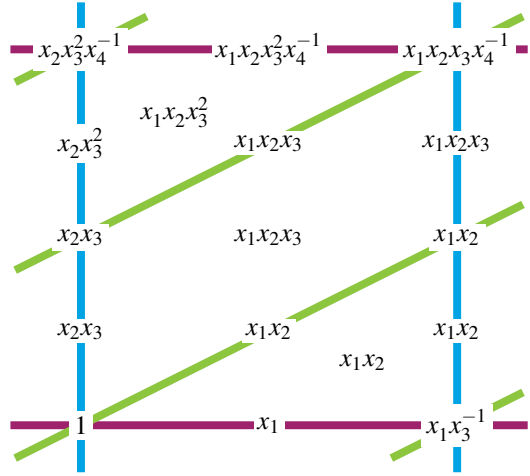
\begin{figure}[ht]
  \centering
  \begin{tikzpicture}[xscale=1.25, yscale=1.25, line width=3pt]
  \usetikzlibrary{calc, decorations.pathreplacing,shapes.misc, arrows.meta}
  \usetikzlibrary{decorations.pathmorphing}
  
  \tikzstyle{fuzz}=[RedViolet]
  \tikzstyle{fuzzB}=[Cerulean]
  \tikzstyle{fuzzP}=[LimeGreen]

  \coordinate  (origin) at (0,0);
  \coordinate (br) at (4,0);
  \coordinate (tr) at (4,4);
  \coordinate (tl) at (0,4);
  \coordinate (ml) at (0,2);
  \coordinate (mr) at (4,2);
  
  \coordinate (hoff) at (0.75,0); 
  \draw[RedViolet] ($(origin)-(hoff)$) -- ($(br)+(hoff)$);
  \draw[RedViolet] ($(tl)-(hoff)$)--($(tr)+(hoff)$);

  \coordinate (voff) at (0,0.5);
  \draw[Cerulean] ($(origin)-(voff)$) -- ($(tl)+(voff)$);
  \draw[Cerulean] ($(br)-(voff)$) -- ($(tr)+(voff)$);

  \coordinate (diagoff) at (0.7,0.35);
  \draw[LimeGreen] ($(br)-(diagoff)$) -- ($(br)+(diagoff)$);
  \draw[LimeGreen] ($(origin)-(diagoff)$) -- ($(mr)+(diagoff)$);
  \draw[LimeGreen] ($(ml)-(diagoff)$) -- ($(tr)+(diagoff)$);
  \draw[LimeGreen] ($(tl)-(diagoff)$) -- ($(tl)+(diagoff)$);

  \node[inner sep = 2pt, fill=white] at (origin) {\footnotesize $1$};
  \node[inner sep = 2pt, fill=white] at (ml) {\footnotesize $x_2 x_3$};
  \node[inner sep = 2pt, fill=white] at (tl) {\footnotesize $x_2^{} x_3^2 x_4^{-1}$};
  \node[inner sep = 2pt, fill=white] at (br) {\footnotesize $x_1^{} x_3^{-1}$};
  \node[inner sep = 2pt, fill=white] at (mr) {\footnotesize $x_1 x_2$};
  \node[inner sep = 2pt, fill=white] at (tr)
  {\footnotesize $x_1^{} x_2^{} x_3^{} x_4^{-1}$};

  \begin{scope}[inner sep = 1.5pt]
    \node[fill=white] at ($(origin)!0.5!(br)$) {\footnotesize $x_1$};
    \node[fill=white] at ($(tl)!0.5!(tr)$) {\footnotesize
      $x_1^{} x_2^{} x_3^2 x_4^{-1}$};
    \node[fill=white] at ($(origin)!0.5!(ml)$) {\footnotesize $x_2x_3$};
    \node[fill=white] at ($(ml)!0.5!(tl)$) {\footnotesize $x_2^{} x_3^2$};
    \node[fill=white] at ($(br)!0.5!(mr)$) {\footnotesize $x_1 x_2$};
    \node[fill=white] at ($(mr)!0.5!(tr)$) {\footnotesize $x_1 x_2 x_3$};
    \node[fill=white] at ($(origin)!0.5!(mr)$) {\footnotesize $x_1 x_2$};
    \node[fill=white] at ($(ml)!0.5!(tr)$) {\footnotesize $x_1 x_2 x_3$};
    
    \node[fill=white] at (3,2/3) {\footnotesize $x_1 x_2$};
    \node[fill=white] at (2,2) {\footnotesize $x_1 x_2 x_3$};
    \node[fill=white] at (1,10/3) {\footnotesize $x_1^{} x_2^{} x_3^2$};
  \end{scope}
\end{tikzpicture}
  \vspace{-0.75em}
  \caption{A compatible $\ZZ^4$-stratification for the closed embedding of
    the identity point into the second Hirzebruch surface.}
  \label{f:H2lab}
\end{figure}

In addition to the cell structure on the space $\RL$, we also regard the free
abelian group $\ZZ^n$ as a topological space.  Since the set $\ZZ^n$ is a
poset under the componentwise order, it carries the \emph{Alexandrov
  topology}: a subset $U$ of $\ZZ^n$ is open if, whenever
$\bm{u} \leqslant \bm{v}$ and $\bm{u} \in U$, it follows that $\bm{v} \in U$.
Building on \cite{Lur17}*{Definition~A.5.1}, we make the following definition.

\begin{definition}
  \label{d:strat}
  Let $L$ be a lattice in $\ZZ^n$ and assume that the space $\RL$ is endowed
  with the structure of an $L$-equivariant cell complex. A continuous map
  $\psi \colon \RL \to \ZZ^n$ is a \emph{compatible $\ZZ^n$-stratification} if
  it is constant on each open cell in $\RL$ and satisfies
  $\psi(\bm{p} + \bm{v}) = \psi(\bm{p}) + \bm{v}$ for all points
  $\bm{p} \in \RL$ and all lattice points $\bm{v} \in L$. For each open cell
  $\sigma \subset \RL$, define $\psi(\sigma) \colequal \psi(\bm{p})$ for any
  point $\bm{p} \in \sigma$; this is well-defined because $\psi$ is constant
  on each cell.
\end{definition}

Since $\ZZ^n$ is endowed the Alexandrov topology, continuity of $\psi$ is
equivalent to requiring that, for each lattice point $\bm{u} \in \ZZ^n$, the
subset $(\RL)_{\leqslant \bm{u}} \colequal \set{\bm{p} \in \RL \mathrel{\big|}
  \psi(\bm{p}) \leqslant \bm{u}}$ is closed.  The second condition defining
$\psi$ implies that it is compatible with lattice translation. By design, a
compatible $\ZZ^n$-stratification $\psi \colon \RL \to \ZZ^n$ is determined by
its values on a fundamental domain for the translation action of the lattice
$L$ on $\RL$.  In our applications, the topological space $\RL$ typically has
the coarsest cell structure that allows a given map $\psi$ to define a
compatible $\ZZ^n$-stratification.

We recover the labeled cell complexes in \cite{MS05}*{Definition~4.2} as a
special case.

\begin{example}[Least common multiple stratifications]
  \label{e:lcm}
  A compatible $\ZZ^n$-stratification $\psi \colon \RL \to \ZZ^n$ is defined
  by assigning a lattice vector in $\ZZ^n$ for each vertex in a fundamental
  domain.  Since a compatible stratification commutes with translation by
  lattice points in $L$, the value $\psi$ at every vertex in $\RL$ is
  determined by this assignment.  For any higher-dimensional open cell
  $\sigma$, the value $\psi(\sigma)$ is the componentwise maximum of the
  lattice vectors assigned to the vertices in the closure of $\sigma$.
\end{example}

The following tangible geometric example serves to further highlight
\Cref{d:strat}.

\begin{example}[Compatible $\ZZ^4$-stratification for an embedding into
  toric surface]
  \label{e:Hceil}
  As in \Cref{e:H2}, consider the closed embedding $\varphi \colon Y \to X$ of
  the identity point $Y$ into the second Hirzebruch surface $X$.  The map
  $\psi \colon \RL \to \ZZ^{4}$ defined by
  $\psi(\bm{p}) \colequal \ceil{\bm{p}}$ for all $\bm{p} \in \RL \subset \RR^4$
  is a compatible $\ZZ^4$\nobreakdash-stratification. Let
  $\iota_\RR \colon \RR^2 \to \RR^4$ be the linear map induced by the
  inclusion $\iota \colon \ZZ^2 \to \ZZ^4$.  For all
  $(q_1, q_2) \in \RR^2 \cong \RL$, we have
  \[
    \psi \bigl( \iota_{\RR}(q_1,q_2) \bigr)
    = \ceil{(q_1,q_2,-q_1+2q_2,-q_2)}
    = (\ceil{q_1}, \ceil{q_2}, \ceil{-q_1+2q_2}, \ceil{-q_2}) \, .
  \]
  In \Cref{f:H2lab}, we display the monomial with exponent vector
  $\psi(\sigma)$ at the center of the cell $\sigma$.
\end{example}

We record one technical feature of compatible $\ZZ^n$-stratifications.

\begin{lemma}
  \label{l:lab}
  Let $\psi \colon \RL \to \ZZ^n$ be a compatible $\ZZ^n$-stratification.  For
  any open cells $\sigma$ and $\tau$ in the space $\RL$, the inclusion
  $\tau \subseteq \overline{\sigma}$ implies that
  $\psi(\tau) \leqslant \psi(\sigma)$.
\end{lemma}

\begin{proof}
  In the Alexandrov topology, the subset
  $\psi(\sigma) - \NN^{n} \colequal \set{\psi(\sigma) - \bm{u} \in \ZZ^n
    \mathrel{\big|} \bm{u} \in \NN^{n}}$ is a closed set.  As the map $\psi$
  is continuous, the preimage
  $\psi^{-1} \kern-1.0pt \bigl( \psi(\sigma) - \NN^n \bigr)$ is closed and
  contains $\sigma$.  The hypothesis $\tau \subseteq \overline{\sigma}$ shows
  that the preimage also contains $\tau$. We deduce that
  $\psi(\tau) \leqslant \psi(\sigma)$.
\end{proof}

\subsection*{Homological algebra from the universal cover}

In algebraic topology, the cell structure on the space $\RL$ leads to a
complex of vector spaces over the field $\kk$. The map $\psi \colon \RL \to
\ZZ^{n}$ allows one to transform this into a graded complex of modules over a
polynomial ring.  Recall that $d \colequal \dim \RL$.

For any open cells $\sigma, \tau \subset \RL$, let $\varepsilon(\sigma, \tau)$
be the \emph{incidence number} (also known as the \emph{topological degree})
of the composition of the attaching map of $\sigma$ with the quotient map
collapsing the complement of $\tau$ to a point; see \cite{Hat02}*{p.~140}. The
resulting map is a continuous map between spheres, and its topological degree
measures how many times the boundary sphere of $\sigma$ wraps around the
sphere corresponding to $\tau$; see \cite{Hat02}*{p.~134}. The integer
$\varepsilon(\sigma, \tau)$ depends implicitly on the orientations chosen for
the cells. In particular, we have $\varepsilon(\sigma, \tau) = 0$ unless
$\dim \sigma = 1 + \dim \tau$ and $\tau$ is contained in the closure
$\overline{\sigma}$. The $L$-equivariance of the cell structure implies that
$\varepsilon(\sigma + \bm{v}, \tau + \bm{v}) = \varepsilon(\sigma, \tau)$ for
any lattice point $\bm{v} \in L$. When the cell structure on $\RL$ is regular,
the incidence numbers take values only in $\set{-1, 0, 1}$; see
\cite{MS05}*{\S4.1}.

The Cellular Boundary Formula in \cite{Hat02}*{\S2.2} establishes that the
\emph{cellular $\kk$-complex} of $\RL$ is
\[
  0
  \longleftarrow C_0(\RL; \kk)
  \xleftarrow{\quad \partial_1 \quad} C_1(\RL; \kk)
  \xleftarrow{\quad \partial_2 \quad} \dotsb
  \xleftarrow{\quad \partial_d \quad} C_{d}(\RL; \kk)
  \longleftarrow 0
\]
with
\vspace{-1.0em}
\begin{align*}
  C_{i}(\RL; \kk)
  &\colequal \bigoplus_{\begin{subarray}{c}
      \sigma \subset \RL \\
      \dim \sigma = i
    \end{subarray}}
  \kk \, \sigma \, ,
  &&\text{and}
  & \partial \sigma
  &= \sum_{\tau \subset \RL} \varepsilon(\sigma, \tau) \, \tau \, ,
\end{align*}
where each the open cell $\sigma$ (or $\tau$) is identified with a generator
of the corresponding summand of the direct sum $C(\RL; \kk)$.

To generalize the cellular free complexes in \cite{BS98}*{\S1} and
\cite{MS05}*{Definition~4.3}, we use a compatible $\ZZ^n$-stratification
$\psi \colon \RL \to \ZZ^n$ to homogenize the cellular $\kk$-complex
$C(\RL; \kk)$.  Let $S \colequal \kk[x_1, x_2, \dotsc, x_n]$ denote the
$\ZZ^n$-graded polynomial ring with
$\deg(\bm{x}^{\bm{u}}) \colequal \bm{u} \in \ZZ^n$.  For any open cell
$\sigma \subset \RL$, let
$S \, \sigma \cong S \bigl( \kern-0.5pt - \psi(\sigma) \kern-1.0pt \bigr)$ be
the free $\ZZ^n$-graded $S$-module with a unique generator, also denoted by
$\sigma$, in degree $\psi(\sigma)$.

\begin{definition}
  \label{d:comp}
  Let $L$ be a lattice in $\ZZ^n$.  For any compatible $\ZZ^n$-stratification
  $\psi \colon \RL \to \ZZ^n$, the \emph{cellular free $S$-complex} $F_\psi$
  is the $\ZZ^n$-graded $S$-complex
  \[
    0
    \longleftarrow (F_{\psi})_0
    \xleftarrow{\quad \partial_1 \quad} (F_{\psi})_1
    \xleftarrow{\quad \partial_2 \quad} \dotsb
    \xleftarrow{\quad \partial_d \quad} (F_{\psi})_{d}
    \longleftarrow 0
  \]
  with
  \vspace{-1.0em}
  \begin{align*}
    (F_{\psi})_i
    &= \bigoplus_{\begin{subarray}{c}
        \sigma \subset \RL \\
        \dim \sigma = i
      \end{subarray}}
    S \, \sigma \, ,
    &&\text{and}
    & \partial \sigma
    &= \sum_{\tau \subset \RL} \varepsilon(\sigma, \tau) \,
      \bm{x}^{\psi(\sigma) - \psi(\tau)} \, \tau \, .
  \end{align*}
  Moreover, the monomial $S$-module $M_\psi$ is defined to be
  \[
    M_{\psi} \colequal S \cdot \set{ \smash{\bm{x}^{\psi(\bm{p})}}
      \mathrel{\big|} \bm{p} \in \RL} \subseteq \kk[x_{1}^{\pm 1}, x_2^{\pm 1}, \dotsc, x_{n}^{\pm 1}].
  \]
\end{definition}

Our earlier technical observation clarifies some aspects of this definition.

\begin{remark}
  \label{r:gens}
  \Cref{l:lab} confirms that all the entries
  $\smash{\varepsilon(\sigma, \tau) \, \bm{x}^{\psi(\sigma) - \psi(\tau)}}$ in
  the matrix representing the differential of $F_{\psi}$ belong to
  $S$. Moreover, for any open cell $\sigma$ in $\RL$ and any vertex $\bm{p}$
  in the closure $\overline{\sigma}$, \Cref{l:lab} guarantees that
  $\smash{\bm{x}^{\psi(\sigma) - \psi(\bm{p})}} \in S$.  Hence, the
  factorization
  $\smash{\bm{x}^{\psi(\sigma)} = \bm{x}^{\psi(\sigma) - \psi(\bm{p})} \,
    \bm{x}^{\psi(\bm{p})}}$ shows that the vertices in $\RL$ form a generating
  set for the monomial $S$-module:
  \[
    M_{\psi} = S \cdot \set{ \smash{\bm{x}^{\psi(\bm{p})}} \mathrel{\big|}
      \text{the point $\bm{p}$ is a vertex in cell complex $\RL$}} \, .
  \]
  In other words, the map that sends a vertex $\bm{p} \in \RL$ to the Laurent
  monomial $\smash{\bm{x}^{\psi(\bm{p})}}$ defines a surjective $\ZZ^n$-graded
  $S$-module homomorphism from $(F_{\psi})_{0}$ onto $M_{\psi}$.
\end{remark}

\begin{example}[Monomial module for an embedding into a toric surface]
  \label{e:recur}
  Following Example~\ref{e:Hceil}, the map $\varphi \colon Y \to X$ is the
  inclusion of the identity point $Y$ in the dense torus into the second
  Hirzebruch surface $X$. A compatible $\ZZ^4$-stratification $\psi \colon \RL
  \to \ZZ^4$ is defined, for all $(q_1, q_2)$ in $\RR^2 \cong \RL$, by $\psi
  \bigl( \kern-0.5pt \iota_\RR(q_1,q_2) \kern-0.5pt \bigr) = ( \ceil{q_1},
  \ceil{q_2}, \ceil{-q_1+2q_2}, \ceil{-q_2})$.  Since applying $\psi$ to the
  vertices gives $\psi \bigl( \iota_\RR \kern-0.5pt (0,0) \kern-0.5pt \bigr) =
  (0, 0, 0, 0)$ and $\psi \bigl(\iota_\RR \kern-0.5pt (0, 0.5) \kern-0.5pt
  \bigr) = (0, 1, 1, 0)$, the monomial $S$-module $M_{\psi}$ is generated by
  the set $\set{\smash{x_{1}^{k} x_{2}^{\ell} x_{3}^{-k+2 \ell}
      x_{4}^{-\ell}}, \smash{x_{1}^{k} x_{2}^{1+\ell} x_{3}^{1-k+2\ell}
      x_{4}^{-\ell}} \mathrel{\big|} k, \ell \in \ZZ}$.
\end{example}

Under suitable hypotheses, the cellular $S$-complex $F_{\psi}$ provides a free
resolution of the monomial $S$-module $M_{\psi}$; compare with
\cite{BS98}*{Proposition~1.2} and \cite{MS05}*{Proposition~4.5}.

\begin{proposition}
  \label{p:res}
  Let $\psi \colon \RL \to \ZZ^n$ be a compatible $\ZZ^n$-stratification. The
  cellular free $S$-complex $F_{\psi}$ is a $\ZZ^n$-graded free resolution of
  $M_{\psi}$ if and only if, for every lattice point $\bm{u} \in \ZZ^{n}$, the
  closed subset
  $(\RL)_{\leqslant \bm{u}} \colequal \set{\bm{p} \in \RL \mathrel{\big|}
    \psi(\bm{p}) \leqslant\bm{u}}$ is acyclic.
\end{proposition}

\Cref{p:strat} exhibits a compatible $\ZZ^n$-stratification that satisfies
this topological condition.

\begin{proof}
  Fix a lattice point $\bm{u} \in \ZZ^{n}$.  Consider the subset
  $(\RL)_{\leqslant \bm{u}} \colequal \set{\bm{p} \in \RL \mathbin{\big|}
    \psi(\bm{p}) \leqslant \bm{u}}$ of $\RL$.  By construction, the degree
  $\bm{u}$ part of the cellular free $S$-complex $F_{\psi}$ is precisely the
  cellular $\kk$\nobreakdash-complex of the subcomplex
  $(\RL)_{\leqslant \bm{u}}$. It follows that $F_{\psi}$ is a free resolution
  of the monomial $S$-module $M_{\psi}$ if and only if
  $H_0 \bigl( \kern-0.5pt (\RL)_{\leqslant \bm{u}}; \kk \bigr) \cong \kk$ when
  $\bm{x}^{\bm{u}} \in M_{\psi}$, and
  $H_i \bigl( \kern-0.5pt (\RL)_{\leqslant \bm{u}}; \kk \bigr) = 0$ when
  $i > 0$ or $\bm{x}^{\bm{u}} \not\in M_{\psi}$.  \Cref{r:gens} establishes
  that $\bm{x}^{\bm{u}} \in M_{\psi}$ if and only if there exists a vertex
  $\bm{p} \in \RL$ such that $\psi(\bm{p}) \leqslant \bm{u}$.  Therefore, the
  $S$-complex $F_{\psi}$ is a free resolution of the $S$-module $M_{\psi}$ if
  and only if every nonempty subcomplex $(\RL)_{\leqslant \bm{u}}$ is acyclic.
\end{proof}

\subsection*{Homological algebra from the topological torus}

Following \cite{BS98}*{\S3} and \cite{MS05}*{\S9.3}, we develop an algebraic
counterpart to the topological relationship between the universal cover $\RL$
and the real torus $\RL \mathbin{\!/\!} L$.  Consider the $\ZZ^n$-graded group
algebra $S[L]$ of the lattice $L$ over the polynomial ring
$S \colequal \kk[x_1, x_2, \dotsc, x_n] = \kk[\NN^{n}]$; it is defined as
\[
  S[L] \colequal S[\bm{z}^{\bm{v}} \mathrel{|} \bm{v} \in L]
  \subseteq S[z_{1}^{\pm 1}, z_{2}^{\pm 1}, \dotsc, z_{n}^{\pm 1}]
  = S[\ZZ^n] \, ,
\]
where $\deg(\bm{x}^{\bm{u}} \, \bm{z}^{\bm{v}}) = \bm{u} + \bm{v} \in
\ZZ^n$. The coefficient ring $S$ of the group algebra $S[L]$ is itself an
$S[L]$-module via the isomorphism
\[
  S[L] \mathbin{\!/\!} \ideal{\bm{z}^{\bm{v}} - 1 \mathrel{|} \bm{v} \in L}
  \cong S \, .
\]
Similarly, the $\ZZ^n$-graded map of $\kk$-algebras from $S[L]$ to the Laurent
polynomial ring
\[
  R \colequal \kk[x_{1}^{\pm 1}, x_2^{\pm 1}, \dotsc, x_{n}^{\pm 1}]
  = \kk[\ZZ^n]
\]
determined by
$\bm{x}^{\bm{u}} \, \bm{z}^{\bm{v}} \mapsto \bm{x}^{\bm{u} + \bm{v}}$ makes
the ring $R$ into a $S[L]$-module.  Furthermore, for any compatible
$\ZZ^n$\nobreakdash-stratification $\psi \colon \RL \to \ZZ^{n}$, the
monomial $S$-module $M_{\psi}$ is an $S[L]$-submodule of $R$ because the
cell complex structure on the space $\RL$ is $L$-equivariant and the map
$\psi$ commutes with lattice translation.  Although $M_{\psi}$ is almost never
a finitely generated $S$-module, it is a finitely generated $S[L]$-module in
all of our applications.

In contrast, the coefficient ring $S$ is \emph{not} an $S[L]$-submodule of the
Laurent polynomial ring $R$ because it is not closed under multiplication by
a Laurent monomial $\bm{z}^{\bm{v}}$ for an arbitrary lattice point $\bm{v} \in L$.
Nevertheless, the cellular free $S$\nobreakdash-complex $F_{\psi}$ is a
$\ZZ^n$-graded $S[L]$-complex.  For any homological index $i$ and any lattice
point $\bm{v} \in L$, multiplication by $\bm{z}^{\bm{v}} \in S[L]$ permutes
the summands in the direct sum $(F_{\psi})_{i}$ by sending the generator of
the summand $S \, \sigma$ to the generator of $S \, \tau$, where
$\tau \colequal \sigma + \bm{v}$; the $L$-equivariance of the underlying cell
complex guarantees that the lattice translate $\tau$ of the open cell $\sigma$
is also an open cell.  Since $\psi(\tau) = \psi(\sigma) + \bm{v}$, this
permutation preserves the $\ZZ^{n}$-grading.  Given the $S[L]$-module
structures on $M_{\psi}$ and $(F_{\psi})_{0}$, \Cref{r:gens} shows that the
canonical surjection from $(F_{\psi})_{0}$ onto $M_{\psi}$ is automatically an
$S[L]$-module homomorphism.  The category of $\ZZ^n$-graded $S[L]$-modules is
our algebraic analog for the universal cover $\RL$.

On the other hand, the category of $(\ZZ^n \mathbin{\!/\!} L)$-graded
$S$-modules serves as our algebraic analog of the real topological torus
$\RL \mathbin{\!/\!}  L$.  Let
\[
  \eta \colon \ZZ^n \to \ZZ^n \mathbin{\!/\!}  L
\]
be the canonical surjective homomorphism of $\ZZ$-modules that sends the
lattice point $\bm{u} \in \ZZ^n$ to the coset
$\eta(\bm{u}) \colequal \set{\bm{u} + \bm{v} \mathrel{|} \bm{v} \in L}$ in the
quotient group $\ZZ^n \mathbin{\!/\!}  L$. By definition, the
$(\ZZ^n \mathbin{\!/\!}  L)$-grading on the polynomial ring $S$ is the
coarsening of its $\ZZ^n$-grading:
\[
  \deg(\bm{x}^{\bm{u}})
  \colequal \eta(\bm{u}) \in \ZZ^{n} \mathbin{\!/\!} L \, .
\]
Likewise, the $(\ZZ^n \mathbin{\!/\!} L)$-grading on $S[L]$ is given by
$\deg(\bm{x}^{\bm{u}} \, \bm{z}^{\bm{v}}) = \eta(\bm{u} + \bm{v}) =
\eta(\bm{u}) \in \ZZ^{n} \mathbin{\!/\!}  L$. Since the ideal
$\ideal{\bm{z}^{\bm{v}} - 1 \mathrel{|} \bm{v} \in L}$ in $S[L]$ is
homogeneous with respect to the $(\ZZ^n \mathbin{\!/\!}  L)$-grading (but not
the $\ZZ^n$\nobreakdash-grading), the $\kk$-algebra homomorphism
$S[L] \to S[L] \mathbin{\!/\!} \!  \ideal{\bm{z}^{\bm{v}} - 1 \mathrel{|}
  \bm{v} \in L} \cong S$ is homogeneous and has degree $\bm{0}$ in the
$(\ZZ^n \mathbin{\!/\!}  L)$-grading. Hence, the corresponding
extension-of-scalars functor
\[
  M \mapsto M \otimes_{S[L]} S
\]
converts every $\ZZ^n$-graded $S[L]$-module into a $(\ZZ^n \mathbin{\!/\!}
L)$-graded $S$-module.

The image of the cellular free $S[L]$-complex $F_{\psi}$ under this
extension-of-scalars functor has a straightforward description.

\begin{remark}
  \label{r:com}
  For each homological index $i$, the summand of the direct sum $(F_{\psi})_i$
  corresponding to the $i$\nobreakdash-dimensional cell $\sigma \subset \RL$ maps
  to the free $(\ZZ^n \mathbin{\!/\!} L)$-graded $S$-module
  $S \, \sigma' \cong S \bigl( \kern-0.5pt - (\eta \circ \psi)(\sigma)
  \kern-1.0pt \bigr)$; the generator of $S \, \sigma'$ corresponds to the
  coset $\sigma' \colequal \set{ \tau \subset \RL \mathrel{\big|}
    \text{$\tau = \sigma + \bm{v}$ for some $\bm{v} \in L$}}$ in
  $\RL\mathbin{\!/\!} L$ and the degree of this generator is
  $\eta \kern-0.5pt \bigl( \kern-0.5pt \psi(\sigma) \kern-1.0pt \bigr)$.  It
  follows that
  \[
    (F_{\psi})_{i} \otimes_{S[L]} S
    = \bigoplus_{\begin{subarray}{c}
        \sigma' \subset \RL \mathbin{/} L \\
        \dim \sigma' = i
      \end{subarray}}
    S \, \sigma' \, .
  \]
  To describe the differential on the free $S$-complex
  $(F_\psi)_{i} \otimes_{S[L]} S$, choose representatives $\sigma$ and $\tau$ in
  $\RL$ for the cosets $\sigma'$ and $\tau'$ in $\RL \mathbin{\!/\!}  L$ that
  lie in the same fundamental domain. It follows that
  \[
    \partial \sigma'
    = \sum_{\tau' \subset \RL \mathbin{/} L} \left(
      \sum_{\bm{v} \in L} \varepsilon(\sigma, \tau + \bm{v}) \,
      \bm{x}^{\psi(\sigma) - \psi(\tau) - \bm{v}} \right) \, \tau' \, .
  \]
  The inner sum is finite: the closed cell $\overline{\sigma}$ contains only
  finitely many open cells, so $\varepsilon(\sigma, \tau + \bm{v}) \neq 0$ for
  only finitely many lattice points $\bm{v} \in L$. Unlike the $\ZZ^n$-graded
  $S[L]$-complex $F_{\psi}$, the differentials in the
  $(\ZZ^n \mathbin{\!/\!}  L)$-graded $S$-complex $F_{\psi} \otimes_{S[L]} S$
  are not necessarily given by monomial matrices.
\end{remark}

The ensuing corollary gives the desired relationship between our two
categories of modules.

\begin{corollary}[\cite{BS98}*{Theorem~3.2}]
  \label{c:res}
  Let $\psi \colon \RL \to \ZZ^{n}$ be a compatible $\ZZ^{n}$-stratification.
  When the cellular free $S[L]$\nobreakdash-complex $F_{\psi}$ is a
  $\ZZ^n$-graded resolution of the $S[L]$\nobreakdash-module $M_{\psi}$, the
  $S$-complex $F_{\psi} \otimes_{S[L]} S$ is a
  $(\ZZ^n \mathbin{\!/\!} L)$-graded free resolution of the $S$-module
  $M_{\psi} \otimes_{S[L]} S$.
\end{corollary}

The categorical arguments in \cite{BS98}*{Theorem~3.2} and
\cite{MS05}*{Theorem~9.17} carry over essentially verbatim to our setting. We
 now outline an alternative proof that implicitly relates cohomology on the universal
cover to that on the torus.

\begin{proof}[Sketch of proof]
  Fix $\alpha$ in the grading group $\ZZ^n \mathbin{\!/\!} L$.  The degree
  $\alpha$ part of the $S$-module $M_\psi \otimes_{S[L]} S$ equals $\kk$ if there
  exists a lattice point $\bm{u} \in \ZZ^{n}$ such that
  $\bm{x}^{\bm{u}} \in M_{\psi}$ and equals $0$ otherwise.  Similarly,
  \Cref{p:res} establishes that the subcomplex $(\RL)_{\leqslant \bm{u}}$ is
  acyclic if $\bm{x}^{\bm{u}} \in M_{\psi}$ and is empty otherwise.  These
  equivalences are independent of the choice of $\bm{u}$ because both
  $M_{\psi}$ and $\RL$ are $L$-equivariant.  It suffices to show that the
  degree $\alpha$ part of the $S$-complex $F_{\psi} \otimes_{S[L]} S$ is
  isomorphic to the degree $\bm{u}$ part of the $S[L]$-complex $F_{\psi}$.
  One verifies that the canonical map from $(F_{\psi})_{\bm{u}}$ to
  $(F_{\psi} \otimes_{S[L]} S)_{\alpha}$ is an isomorphism by describing the
  image and preimage of basis vectors.
\end{proof}

\begin{remark}
  The resolution in \Cref{c:res} leads to formulas for minimal Betti numbers.
  For each $\alpha \in \ZZ^n \mathbin{\!/\!} L$, define $\beta_{i,\alpha}
  \colequal \dim_{\kk} \operatorname{Tor}_{i}^{S}(M_{\psi} \otimes_{S[L]} S,
  \kk)$ to be the $(\ZZ^n \mathbin{\!/\!} L)$-graded minimal Betti number of
  $M_{\psi}\otimes_{S[L]} S$ in homological degree $i$.  Since $\eta \colon
  \ZZ^{n} \to \ZZ^n \mathbin{\!/\!}  L$ is surjective, there is $\bm{u} \in
  \ZZ^{n}$ such that $\eta(\bm{u}) = \alpha$. Assuming $L \cap \NN^{n} =
  \set{\bm{0}}$, \cite{BS98}*{Theorem~1.11} realizes the minimal Betti numbers
  of a monomial module as the relative homology, with coefficients in $\kk$,
  of the pair $\smash{\bigl( \kern-1.0pt (\RL)_{\leqslant \bm{u}}, (\RL)_{<
      \bm{u}} \bigr)}$, and \cite{BS98}*{Corollary~3.3} relates these Betti
  numbers to $\beta_{i,\alpha}$. Equivalently, \cite{FS24}*{Theorem~1.2}
  computes $\beta_{i,\alpha}$ as the compactly supported cohomology of the
  preimage $\smash{\psi^{-1}(\bm{u}) \colequal \set{\bm{p} \in \RL
      \mathrel{\big|} \psi(\bm{p}) = \bm{u}}}$. The compactly supported
  cohomology group may be identifed with a relative cohomology group. Working
  over the field $\kk$, the universal coefficient theorem implies that
  corresponding relative cohomology and relative homology groups have the same
  dimension.
\end{remark}

To leverage \Cref{c:res}, another characterization of the $S$-module
$M_{\psi} \otimes_{S[L]} S$ is needed.  The canonical surjective $\ZZ$-module
homomorphism $\eta \colon \ZZ^n \to \ZZ^n \mathbin{\!/\!} L$ induces a
surjective homomorphism of $\kk$\nobreakdash-algebras from the Laurent
polynomial ring $R = \kk[\ZZ^n]$ to the group algebra
$\kk[\ZZ^n \mathbin{\!/\!} L]$: for any lattice point $\bm{u} \in \ZZ^n$, the
monomial $\bm{x}^{\bm{u}} \in R$ maps to the monomial
$\smash{\bm{x}^{\eta(\bm{u})}} \in \kk[\ZZ^n\mathbin{\!/\!} L]$. Expanding on
this identification gives the following description of
$M_{\psi} \otimes_{S[L]} S$.

\begin{proposition}
  \label{p:gens}
  Let $\psi \colon \RL \to \ZZ^{n}$ be a compatible $\ZZ^{n}$-stratification.
  The module $M_\psi \otimes_{S[L]} S$ is the $S$\nobreakdash-submodule of the
  group algebra $\kk[\ZZ^n \mathbin{\!/\!} L]$ generated by
  \[
    \set{\smash{\bm{x}^{\eta \circ \psi (\bm{p})}} \in \kk[\ZZ^n \mathbin{\!/\!} L]
      \mathrel{\big|} \text{the point $\bm{p}$ is a vertex in the cell complex
        $\RL$}} \, .
  \]
\end{proposition}

\begin{proof}
  Since the extension-by-scalars functor is right exact, there exists a
  surjective $S$-module homomorphism from $R \otimes_{S[L]} S$ to
  $\kk[\ZZ^n \mathbin{\!/\!} L] \otimes_{S[L]} S$.  We start by focusing on
  the target module.  Using the $S[L]$-module structure on $R$, the surjective
  $\kk$-algebra homomorphism from $R$ to $\kk[\ZZ^n \mathbin{\!/\!} L]$ endows
  the group algebra with an $S[L]$-module structure.  In particular, for any
  lattice point $\bm{v} \in L$, the product
  $\bm{z}^{\bm{v}} \cdot \bm{x}^{\bm{u}} = \bm{x}^{\bm{u} + \bm{v}} \in R$
  maps to
  $\bm{x}^{\eta(\bm{u} + \bm{v})} = \bm{x}^{\eta(\bm{u})} \in \kk[\ZZ^n
  \mathbin{\!/\!} L]$.  Hence, multiplication by the monomial
  $\bm{z}^{\bm{v}} \in S[L]$ on the group algebra
  $\kk[\ZZ^n \mathbin{\!/\!}  L]$ is trivial.  We deduce that
  $\kk[\ZZ^n \mathbin{\!/\!} L] \otimes_{S[L]} S \cong \kk[\ZZ^n
  \mathbin{\!/\!} L]$.

  Returning to the source module, the ring $S$ becomes an $S[L]$-module via
  $S \cong S[L] \mathbin{\!/\!} \! \ideal{\bm{z}^{\bm{v}} - 1 \mathrel{|}
    \bm{v} \in L}$.  For any lattice point $\bm{u} \in \ZZ^{n}$ and any
  lattice point $\bm{v} \in L$, it follows that
  \[
    \bm{x}^{\bm{u} + \bm{v}} \otimes 1_{S}
    = \bm{x}^{\bm{u}} \, \bm{z}^{\bm{v}} \otimes 1_{S}
    = \bm{x}^{\bm{u}} \otimes 1_{S}
  \]
  in $R \otimes_{S[L]} S$.  Hence, two monomials in $R$ are equivalent in the
  $S$-module $R \otimes_{S[L]} S$ if and only if their exponent vectors belong
  to the same coset in $\ZZ^n \mathbin{\!/\!} L$.  Having identified their
  monomial bases (as $\kk$-vector spaces), we also deduce that $R
  \otimes_{S[L]} S \cong \kk[\ZZ^n \mathbin{\!/\!} L]$.

  Lastly, \Cref{r:gens} shows that the $S$-module $M_{\psi}$ is the submodule
  generated by the monomials $\smash{\bm{x}^{\psi(\bm{p})} \in R}$ for all
  vertices $\bm{p}$ in the cell complex $\RL$. Since the extension-by-scalars
  functor is exact (see \cite{BS98}*{Theorem~3.2} or
  \cite{MS05}*{Theorem~9.17}), it follows that $\smash{M_\psi \otimes_{S[L]}
    S}$ is the submodule generated by the monomials $\smash{\bm{x}^{\eta \circ
      \psi(\bm{p})}} \in \kk[\ZZ^n \mathbin{\!/\!} L]$ for all vertices
  $\bm{p}$ in the cell complex $\RL$.
\end{proof}

To illustrate a cellular free resolution of $M_{\psi} \otimes_{S[L]} S$, we
revisit the earlier toric example.

\begin{example}[Cellular free resolution for an embedding into a toric
  surface]
  \label{e:HC}
  As in Example~\ref{e:recur}, consider the inclusion of the identity point
  $Y$ of the dense torus into the second Hirzebruch surface $X$ with the
  compatible $\ZZ^4$-stratification arising from the ceiling function. Let
  $S \colequal \kk[x_1, x_2, x_3, x_4]$ denote the Cox ring of $X$ and choose
  the isomorphism $\Pic(X) \cong \ZZ^2$, so that the degree of the variable $x_i$
  is identified with the $i$th column in the matrix
  \[
    \begin{bmatrix*}
      1 & -2 & 1 & 0 \\
      0 & 1 & 0 & 1 \\
    \end{bmatrix*} \,.
  \]
  \Cref{f:twi} displays the twists of $S$ corresponding to each open
  cell and the induced maps on the basis.  The (nonminimal) cellular free
  resolution of the $S$-module $M_\psi \otimes_{S[L]} S$ is
  \[
    \begin{matrix}
      S(0,0) \\[-3pt]
      \oplus \\[-3pt]
      S(1,-1) \\
    \end{matrix}
    \xleftarrow{
      \begin{bsmallmatrix}
	x_4  & x_4 & -x_2 x_3 & -x_1 x_2 & x_3 - x_1 \\
	-x_3 & -x_1 & 1 & 1 & 0
      \end{bsmallmatrix}}
    \begin{matrix}
      S(0,-1)^2 \\[-3pt]
      \oplus \\[-3pt]
      S(1,-1)^2 \\[-3pt]
      \oplus \\[-3pt]
      S(-1,0) \\
    \end{matrix}
    \xleftarrow{
      \begin{bsmallmatrix}
	-x_1 & 1 & 0 \\
	x_3 & -1 & 0 \\
	0 & -x_1 & 1 \\
	0 & x_3 & -1 \\
	-x_4 & 0 & x_2
      \end{bsmallmatrix}}
    \begin{matrix}
      S(-1,-1) \\[-3pt]
      \oplus \\[-3pt]
      S(0,-1) \\[-3pt]
      \oplus \\[-3pt]
      S(1,-1) \\
    \end{matrix}
    \gets 0 \, .
    \qedhere
    \]
\end{example}

\begin{figure}[ht]
  \centering
  \begin{tikzpicture}[xscale=2.0, yscale=2.0]
  \usetikzlibrary{calc, decorations.pathreplacing,shapes.misc, arrows.meta}
  \usetikzlibrary{decorations.pathmorphing}

  \tikzstyle{fuzz}=[red,
  postaction={draw, decorate, decoration={border,
      amplitude=0.75cm,angle=90,segment length=.4cm}}]

  \tikzstyle{fuzzB}=[blue,
  postaction={draw, decorate, decoration={border,
      amplitude=0.75cm,angle=90,segment length=.4cm}}]

  \tikzstyle{fuzzP}=[Thistle,
  postaction={draw, decorate, decoration={border,
      amplitude=0.75cm,angle=90,segment length=.4cm}}]

  \begin{scope}[line width=3pt]
    \coordinate  (origin) at (0,0);
    \coordinate (br) at (4,0);
    \coordinate (tr) at (4,4);
    \coordinate (tl) at (0,4);
    \coordinate (ml) at (0,2);
    \coordinate (mr) at (4,2);

    \coordinate (hoff) at (0.5,0);
    \draw[RedViolet] ($(origin)-(hoff)$) -- ($(br)+(hoff)$);
    \draw[RedViolet] ($(tl)-(hoff)$)--($(tr)+(hoff)$);

    \coordinate (voff) at (0,0.5);
    \draw[Cerulean] ($(origin)-(voff)$) -- ($(tl)+(voff)$);
    \draw[Cerulean] ($(br)-(voff)$) -- ($(tr)+(voff)$);

    \coordinate (diagoff) at (0.5,0.25);
    \draw[LimeGreen] ($(br)-(diagoff)$) -- ($(br)+(diagoff)$);
    \draw[LimeGreen] ($(origin)-(diagoff)$) -- ($(mr)+(diagoff)$);
    \draw[LimeGreen] ($(ml)-(diagoff)$) -- ($(tr)+(diagoff)$);
    \draw[LimeGreen] ($(tl)-(diagoff)$) -- ($(tl)+(diagoff)$);
  \end{scope}

  \begin{scope}[inner sep = 2pt]
    \node[fill=white] (v1) at (0,0) {\footnotesize $S$};
    \node (v1x) at (4,0) {};
    \node (v1y) at (0,4) {};
    \node (v1xy) at (4,4) {};
    \node[fill=white] (v2) at (2,0) {\footnotesize $S(-1,0)$};
    \node (v2y) at (2,4) {};
    \node[fill=white] (v4) at (0,1) {\footnotesize $S(1,-1)$};
    \node (v4x) at (4,1) {};
    \node[fill=white] (v5) at (0,2) {\footnotesize $S(1,-1)$};
    \node (v5x) at (4,2) {};
    \node[fill=white] (v6) at (0,3) {\footnotesize $S(0,-1)$};
    \node (v6x) at (4,3) {};
    \node[fill=white] (v8) at (1.0,3.6) {\footnotesize $S(-1,-1)$};
    \node[fill=white] (v9) at (2,1) {\footnotesize $S(1,-1)$};
    \node[fill=white] (v10) at (2,2) {\footnotesize $S(0,-1)$};
    \node[fill=white] (v11) at (2,3) {\footnotesize $S(0,-1)$};
    \node[fill=white] (v12) at (3.2,0.45) {\footnotesize $S(1,-1)$};
  \end{scope}

  \begin{scope}[inner sep = 1.5pt, arrows = {-Latex[width'=0pt .66, length=8pt]}]
    \draw[very thick] (v2) -- node[midway, fill=white]{\scriptsize $-x_1$} (v1);
    \draw[very thick] (v2) -- node[midway, fill=white]{\scriptsize $x_3$} (v1x);
    \draw[very thick] (v4) -- node[midway, fill=white]{\scriptsize $-x_2x_3$} (v1);
    \draw[very thick] (v4) -- node[midway, fill=white]{\scriptsize $1$} (v5);
    \draw[very thick] (v6) -- node[midway, fill=white]{\scriptsize $-x_3$}(v5);
    \draw[very thick] (v6) -- node[midway, fill=white]{\scriptsize $x_4$}(v1y);
    \draw[very thick] (v9) -- node[midway, fill=white]{\scriptsize $-x_1x_2$}(v1);
    \draw[very thick] (v9) -- node[midway, fill=white]{\scriptsize $1$}(v5x);
    \draw[very thick] (v11) -- node[inner sep = 1pt, midway,fill=white]{\scriptsize
      $-x_1$}(v5);
    \draw[very thick] (v11) -- node[midway, fill=white]{\scriptsize $x_4$}(v1xy);
    \draw[very thick] (v8) -- node[inner sep = 0pt, midway,
    fill=white]{\scriptsize $-x_1$}(v6);
    \draw[very thick] (v8) -- node[inner sep = 0.5pt, midway,
    fill=white]{\scriptsize $-x_4$}(v2y);
    \draw[very thick] (v8) -- node[midway, fill=white]{\scriptsize $x_3$}(v11);
    \draw[very thick] (v12) -- node[inner sep = 0.5pt, midway,
    fill=white]{\scriptsize $x_2$}(v2);
    \draw[very thick] (v12) -- node[midway, fill=white]{\scriptsize $-1$}(v9);
    \draw[very thick] (v12) -- node[midway, fill=white]{\scriptsize $1$}(v4x);
    \draw[very thick] (v10) -- node[midway, fill=white]{\scriptsize $x_3$}(v9);
    \draw[very thick] (v10) -- node[midway, fill=white]{\scriptsize $-x_1$}(v4);
    \draw[very thick] (v10) -- node[midway, fill=white]{\scriptsize $-1$}(v11);
    \draw[very thick] (v10) -- node[midway, fill=white]{\scriptsize $1$}(v6x);
  \end{scope}
\end{tikzpicture}
  \vspace{-0.75em}
  \caption{A diagram encoding the cellular free $S$-complex arising from the
    closed embedding of the identity point into the second Hirzebruch
    surface.}
  \label{f:twi}
  \vspace{-0.5em}
\end{figure}

To further illustrate cellular free resolutions, we present a more substantial
example.

\begin{example}[Cellular free resolution for a diagonal embedding]
  \label{e:diag}
  Let $Y \colequal \PP \bigl( \mathcal{O}_{\PP^1} \oplus
  \mathcal{O}_{\PP^1}(2) \kern-1.0pt \bigr)$ be the second Hirzebruch surface.
  Following \Cref{r:toric} or \cite{BPS01}*{Equation~3.2}, the lattice arising
  from the diagonal embedding $\Delta \colon Y \to X \colequal Y \times
  Y$$\Delta$ is the Lawrence lifting
  \[
    L = \Lambda(M_Y) \colequal
    \smash{\set{(\bm{v}, - \bm{v}) \in \ZZ^4 \times \ZZ^4 \mathrel{\big|}
        \bm{v} \in M_Y \subset \ZZ^{4}}}
    \subset \ZZ^{8} \, .
  \]
  The lattice $L \cong M_{Y}$ has rank $2$, and the matrix of the inclusion
  $\iota \colon \ZZ^2 \cong L \to \ZZ^8$, with respect to the standard bases,
  is%
  \vspace*{-0.75em}
  \[
    \begin{bmatrix*}[r]
       1 &  0 & -1 &  0 & -1 &  0 &  1 & 0 \\[-2pt]
       0 &  1 &  2 & -1 &  0 & -1 & -2 & 1 \\
    \end{bmatrix*}^{\textsf{T}} \, .
  \]
  Let $\iota_\RR \colon \RR^2 \to \RR^8$ be the linear map induced by
  $\iota \colon \ZZ^2 \to \ZZ^8$. The map $\psi \colon \RL \to \ZZ^{8}$
  defined by $\psi(\bm{p}) \colequal \ceil{\bm{p}}$ for all
  $\bm{p} \in \RL \subset \RR^8$ is a compatible
  $\ZZ^8$\nobreakdash-stratification; see \Cref{p:strat}.  For all
  $(q_1, q_2) \in \RR^2$, we have
  \begin{align*}
    \psi \bigl( \iota_{\RR}(q_1,q_2) \bigr)
    &= \ceil{(q_1, q_2, -q_1+2q_2, -q_2, -q_1, -q_2, q_1-2q_2, q_2)} \\
    &= (\ceil{q_1}, \ceil{q_2}, \ceil{-q_1+2q_2}, \ceil{-q_2},
    \ceil{-q_1}, \ceil{-q_2}, \ceil{q_1-2q_2}, \ceil{q_2}) \, .
  \end{align*}
  The Cox ring of $X$ is
  $S \colequal \kk[x_1, x_2, x_3, x_4, y_1, y_2, y_3, y_4]$, where
  $\Pic(X) = \Pic(Y) \oplus \Pic(Y) \cong \ZZ^{2} \oplus \ZZ^{2}$; see
  \Cref{e:HC}.  Hence, the (minimal) cellular free resolution of the
  $S$-module $M_\psi \otimes_{S[L]} S$ is
  \[
    \begin{matrix}
      S(0, 0, \; 0, 0) \\[-3pt]
      \oplus \\[-3pt]
      S(1, -1, \; 1, -1) \\
    \end{matrix}
    \xleftarrow{\quad \partial_1 \quad}
    \begin{matrix}
      S(0, -1, \; 1, -1)^2 \\[-3pt]
      \oplus \\[-3pt]
      S(1, -1, \; 0, -1)^2 \\[-3pt]
      \oplus \\[-3pt]
      S(-1, 0, \; -1, 0) \\
    \end{matrix}
    \xleftarrow{\quad \partial_2 \quad}
    \begin{matrix}
      S(-1, -1, \; 1, -1) \\[-3pt]
      \oplus \\[-3pt]
      S(0, -1, \; 0, -1) \\[-3pt]
      \oplus \\[-3pt]
      S(1, -1, \; -1, -1) \\
    \end{matrix}
    \gets 0 \, ,
  \]
  with%
  \vspace{-1em}
  \[
    \partial_1
    = \begin{bmatrix*}
      x_4 y_2 y_3 & x_4 y_1 y_2 & -x_2 x_3 y_4
      & -x_1 x_2 y_4 & x_3 y_1 - x_1 y_3 \\
      -x_3 & -x_1 & y_3 & y_1 & 0 \\
    \end{bmatrix*}
    \quad \text{and} \quad
    \partial_2
    = \begin{bmatrix*}
      -x_1 & y_1 & 0 \\
      x_3 & -y_3 & 0 \\
      0 & -x_1 & y_1 \\
      0 & x_3 & -y_3 \\
      -x_4y_2 & 0 & x_2y_4
    \end{bmatrix*} \, .
  \]
  Taking associated sheaves, one obtains a locally-free resolution of the
  $\mathcal{O}_{X}$-module $\Delta_{*}^{} \mathcal{O}_{Y}$.
\end{example}

\section{The ceiling stratification}
\label{s:ceil}

\noindent
This section exhibits a combinatorial free resolution for the integral closure
of the comodule of a toric ideal.  More precisely, for any prime lattice ideal
$I_L$ in the polynomial ring $S \colequal \kk[x_1, x_2, \dotsc, x_n]$, we
produce a compatible $\ZZ^n$-stratification $\psi \colon \RL \to \ZZ^n$ such
that the cellular free $S$-complex $F_{\psi} \otimes_{S[L]} S$ is a
$(\ZZ^n \mathbin{\!/\!} L)$-graded resolution of the integral closure of the
$S$-module $S \mathbin{\!/\!} I_{L}$.

The lattice $L$ is still a subgroup of $\ZZ^n$.  We further assume that $L$ is
\emph{saturated}, which means
$L = \set{\bm{u} \in \ZZ^n \mathrel{\big|} \text{$k \, \bm{u} \in L$ for some
    $k \in \NN$}}$.  This saturated hypothesis ensures that the \emph{lattice
  ideal}
\[
  I_{L} \colequal \ideal{\bm{x}^{\bm{u}} - \bm{x}^{\bm{v}} \mathrel{\big|}
    \text{$\bm{u} - \bm{v} \in L$ for all $\bm{u}, \bm{v} \in \NN^{n}$}}
\]
is a prime ideal in the polynomial ring $S$; see \cite{MS05}*{Theorem~7.4 and
  Definition~7.24}.  A prime lattice ideal is also known as a \emph{toric
  ideal}.

In this context, the function on the space $\RL$ that assigns the least
integer greater than or equal to each coordinate deserves a formal name.

\begin{definition}
  \label{d:ceil}
  For any lattice $L \subseteq \ZZ^n$, the \emph{ceiling stratification} is the map
  $\psi \colon \RL \to \ZZ^n$ defined by
  $\psi(\bm{p}) \colequal \ceil{\bm{p}} = (\ceil{p_1}, \ceil{p_2}, \dotsc,
  \ceil{p_n})$ for any point $\bm{p} \in \RL \subseteq \RR^n$.
\end{definition}

\Cref{e:Hceil} presents a ceiling stratification. In the special case of toric
varieties, this stratification appears implicitly in \cite{Bon06} and
explicitly in \cite{HHL24}*{\S2.3}.

The standard cell structure is, in fact, the coarsest cell structure on the
space $\RL$ for which the coordinatewise ceiling function is constant on open
cells.

\begin{proposition}
  \label{p:strat}
  Let $\psi \colon \RL \to \ZZ^n$ be the ceiling stratification for a lattice
  $L \subseteq \ZZ^n$. When $\RL$ has its standard cell structure, the map
  $\psi$ is a compatible $\ZZ^n$-stratification, and the cellular free
  $S$-complex $F_{\psi}$ is a $\ZZ^n$-graded free resolution of the
  $\ZZ^n$-graded $S$-module $M_{\psi}$.
\end{proposition}

\begin{proof}
  For each real number $p$ and integer $v$, the ceiling function satisfies
  $\ceil{p+v} = \ceil{p} + v$. It follows that, for all points $\bm{p}$ in
  $\RL$ and all lattice points $\bm{v} \in L$, we have
  \[
    \psi(\bm{p} + \bm{v})
    = \ceil{\bm{p} + \bm{v}}
    = \ceil{\bm{p}} + \bm{v}
    = \psi(\bm{p}) + \bm{v}
    \, .
  \]
  Thus, the ceiling stratification is compatible with lattice translation.

  The set $\ZZ^{n}$ is a poset under the componentwise order. Hence, for any
  lattice point $\bm{u} \in \ZZ^n$, the subset
  $(\RL)_{\leqslant \bm{u}} \colequal \set{\bm{p} \in \RL \mathrel{\big|}
    \psi(\bm{p}) = \ceil{\bm{p}} \leqslant \bm{u}}$ is the intersection of
  $\RL$ with $\bm{u} + (\RR_{\leqslant 0})^n$.  Because the space $\RL$ has
  the standard cell structure, the subset $(\RL)_{\leqslant \bm{u}}$ is
  closed.  It follows that the ceiling stratification is a continuous map that
  is constant on open cells.  The subset $(\RL)_{\leqslant \bm{u}}$ is a
  polyhedron and, thereby, contractible or empty.  Hence, \Cref{p:res}
  establishes that the cellular free $S$-complex $F_\psi$ is a resolution of
  the $S$-module $M_\psi$.
\end{proof}

For a ceiling stratification, the monomial $S$-module
$M_{\psi} \otimes_{S[L]} S$ has an appealing reinterpretation.

\begin{proposition}
  \label{p:nor}
  For the ceiling stratification $\psi \colon \RL \to \ZZ^n$ of a saturated
  lattice $L \subseteq \ZZ^n$, the $S$-module $M_{\psi} \otimes_{S[L]} S$ is the
  integral closure of the $S$-module $S \mathbin{\!/\!} I_{L}$.
\end{proposition}

\begin{proof}
  Consider the quotient semigroup
  $A \colequal \NN^n \mathbin{\!/\!} \mathbin{\sim_L}$, where the equivalence
  relation $\mathbin{\sim_L}$ on $\NN^n$ is defined by
  $\bm{u} \mathbin{\sim_{L}} \bm{v}$ if and only if $\bm{u} - \bm{v} \in
  L$. The canonical surjection $\eta \colon \ZZ^n \to \ZZ^n \mathbin{\!/\!} L$
  induces a surjective $\kk$-algebra map $S \to \kk[A]$ with kernel $I_L$, so
  $S \mathbin{\!/\!} I_L \cong \kk[A]$; see \cite{MS05}*{Theorem~7.3}. The
  integral closure of $\kk[A]$ is generated, as a $\kk$-algebra, by the
  saturation $\overline{A}$; see \cite{MS05}*{Theorem~7.25}.  We claim that
  $\overline{A}$ is generated over $A$ by the elements $\eta(\ceil{\bm{p}})$,
  where $\bm{p}$ runs over the vertices of $\RL$.

  To see this, let $\bm{p} \in \RL$ be a vertex. We first establish that
  $\eta(\ceil{\bm{p}}) \in \overline{A} = (\ZZ^n \mathbin{\!/\!} L) \cap
  (\RR_{\geqslant 0} \, A)$. Since the target of $\eta$ is
  $\ZZ^n \mathbin{\!/\!}  L$, it suffices to prove that
  $\eta(\ceil{\bm{p}}) \in (\RR_{\geqslant 0} \, A)$. Consider the
  $\RR$-linear extension
  $\eta_{\RR} \colequal \eta \otimes_\ZZ \id_\RR \colon \ZZ^n \otimes_\ZZ \RR
  \to (\ZZ^n \mathbin{\!/\!} L) \otimes_\ZZ \RR$.  Since $\bm{p} \in \RL$, we
  see that $\eta_{\RR}(\bm{p}) = \bm{0}$.  Moreover, the difference
  $\ceil{\bm{p}} - \bm{p}$ has nonnegative entries, so
  $\eta_{\RR}(\ceil{\bm{p}} - \bm{p}) \in (\RR_{\geqslant 0} \, A)$. Thus, we
  obtain
  \begin{align*}
    \eta(\ceil{\bm{p}})
    = \eta(\ceil{\bm{p}}) - \eta_\RR(\bm{p})
    = \eta_{\RR}(\ceil{\bm{p}} - \bm{p})
    \in (\RR_{\geqslant 0} \, A) \, .
  \end{align*}

  We next establish that the inclusion
  $\overline{A} \subseteq A + \smash{\set{\eta(\ceil{\bm{p}}) \mathrel{\big|}
      \text{the point $\bm{p}$ is a vertex in $\RL$}}}$. Let
  $\bm{a} \in \overline{A}$. By construction, there exist
  $\bm{q} \in \RR_{\geqslant 0}^n$ and $\bm{u} \in \ZZ^n$ such that
  $\bm{a} = \eta_{\RR}(\bm{q}) = \eta(\bm{u})$.  Set
  $\bm{b} \colequal \eta(\floor{\bm{q}}) \in A$ and
  $\bm{p} \colequal \bm{u} - \bm{q}$.  Since the kernel of $\RR$-linear map
  $\eta_\RR$ is $\RL$ and
  \[
    \eta_\RR(\bm{p})
    = \eta_\RR(\bm{u} - \bm{q})
    = \eta(\bm{u}) - \eta_\RR(\bm{q})
    = \bm{a} - \bm{a}
    = \bm{0} \, ,
  \]
  it follows that $\bm{p} \in \RL$ and
  $\eta(\ceil{\bm{p}}) = \eta( \ceil{\bm{u} - \bm{q}}) = \eta(\bm{u}) + \eta(
  \ceil{- \bm{q}}) = \eta(\bm{u}) - \eta( \floor{\bm{q}}) = \bm{a} - \bm{b}$.
  If $\bm{p} \in \RL$ is not a vertex, then it is contained in the relative
  interior of a cell of $\RL$.  In particular, there exists a
  minimal-dimensional integral unit parallelepiped $P$ in $\RR^n$ containing
  $\bm{p}$ such that $P \cap \RL$ coincides with the minimal-dimensional closed
  cell containing $\bm{p}$. For any vertex $\bm{p}'$ with
  $\bm{p}' \in \RL \cap P$, we have
  $\ceil{\bm{p}} - \ceil{\bm{p}'} \in \NN^n$. It follows that
  $\bm{c} \colequal \eta(\ceil{\bm{p}} - \ceil{\smash{\bm{p}'}})$ is in $A$
  and
  \[
    \eta(\ceil{\smash{\bm{p}'}})
    = \eta(\ceil{\bm{p}}) - \bm{c}
    = \bm{a} - \bm{b} - \bm{c} \, .
  \]
  We deduce that $\bm{a} = \bm{b} + \bm{c} + \eta(\ceil{\bm{p}'})$, where
  $\bm{b} + \bm{c} \in A$ and $\bm{p}'$ is a vertex in $\RL$.

  Finally, Proposition~\ref{p:gens} establishes that the $S$-module
  $\smash{M_\psi \otimes_{S[L]} S}$ is generated in $\kk[\ZZ^n/L]$ by the
  monomials $\bm{x}^{\eta \circ \psi(\bm{p})} = \bm{x}^{\eta(\ceil{\bm{p}})}$,
  where $\bm{p}$ runs over the vertices of $\RL$. Since saturation of a
  semigroup corresponds to integral closure of its semigroup ring, we
  conclude that
  \[
    M_\psi \otimes_{S[L]} S
    \cong \kk[\overline{A}]
    = \overline{\kk[A]} \cong \overline{S/I_L}
    \, . \qedhere
  \]
\end{proof}

Combined with \Cref{c:res}, this leads to the desired cellular free resolution
for $\overline{S/I_L}$.

\begin{theorem}
  \label{t:resSat}
  For the ceiling stratification $\psi \colon \RL \to \ZZ^n$ of a saturated
  lattice $L \subseteq \ZZ^n$, the cellular free $S$-complex
  $F_{\psi} \otimes_{S[L]} S$ is a resolution of the integral closure of the
  $S$-module $S \mathbin{\!/\!} I_{L}$.
\end{theorem}

\begin{proof}
  \Cref{c:res} and \Cref{p:strat} establish that the cellular free $S$-complex
  $F_{\psi} \otimes_{S[L]} S$ is a $(\ZZ^n \mathbin{\!/\!}  L)$-graded
  resolution of the $S$-module $M_\psi \otimes_{S[L]} S$, and \Cref{p:nor}
  shows that the $S$-module $M_\psi \otimes_{S[L]} S$ is isomorphic to the
  integral closure of the $S$-module $S \mathbin{\!/\!}  I_{L}$.
\end{proof}

\begin{proof}[Proof of \Cref{t:main}]
  This result is a restatement of \Cref{t:resSat}.
\end{proof}

\begin{example}[Integral closure of a concrete toric ideal]
  As in \Cref{e:diag}, consider the diagonal embedding of the second
  Hirzebruch surface and the compatible stratification
  $\psi \colon \RL \to \ZZ^{8}$ defined by
  $\psi(\bm{p}) = \ceil{\bm{p}}$.  The corresponding toric
  ideal in $S \colequal \kk[x_1, x_2, x_3, x_4, y_1, y_2, y_3, y_4]$ is
  \[
    I_L
    = \ideal{x_3^{} y_1^{} - x_1^{} y_3^{}, \;
      x_4^{}y_2^{}y_3^2 - x_2^{} x_3^2 y_4^{}, \;
      x_4^{} y_1^{} y_2^{} y_3^{} - x_1^{} x_2^{} x_3^{} y_4^{}, \;
      x_4^{} y_1^2 y_2^{} - x_1^2 x_2^{} y_4^{}} \, .
  \]
  The element $x_2^{} x_3^{} y_3^{-1} y_4^{}$ of the fraction field of $S/I_L$
  is a root of the monic polynomial $w^2 - x_2 x_4 y_2 y_4$ and, thus,
  contained in the integral closure.  This element corresponds to the
  half-integer vertex:
  \[
    \psi \bigl( \iota_{\RR} (0, 0.5) \kern-1.0pt \bigr)
    = \ceil{( 0, 0.5, 1, - 0.5, 0, - 0.5, -1, 0.5)}
    = (0, 1, 1, 0, 0, 0, -1, 1) \, .  \qedhere
  \]
\end{example}

\begin{remark}[Unimodular case]
  When the lattice $L \subseteq \ZZ^m$ is unimodular, the ceiling
  stratification on the Lawrence lift $\Lambda(L) \subseteq \ZZ^{2m}$ of $L$
  agrees with the labeling in \cite{BPS01}*{\S3}.  Hence, \Cref{t:resSat}
  recovers, as a special case, the resolution of the diagonal via the Lawrence
  ideal for a unimodular toric variety in \cite{BPS01}*{\S6}.
\end{remark}

\begin{remark}[Eliminating the saturated assumption]
  \label{r:nonSatDetails}
  While the lattice $L$ is assumed to be saturated in \Cref{t:resSat}, the
  conclusion of the theorem remains valid even without this assumption.
  However, some care is required in formulating the correct generalization.
  When $L$ is not saturated, the quotient ring $S \mathbin{\!/\!}  I_{L}$
  need not be a domain. Consequently, the integral closure of
  $S \mathbin{\!/\!}  I_{L}$ should be taken within its total ring of
  fractions.

  When $L$ is not saturated, the quotient semigroup $A \colequal \NN^n
  \mathbin{\!/\!} \mathbin{\sim_L}$ is not affine, so one must refine the
  notion of saturation.  The group completion $\ZZ^n \mathbin{\!/\!} L$ of $A$
  is a finitely generated abelian group, decomposing as the direct sum of a
  free abelian group and a torsion group: $\ZZ^n \mathbin{\!/\!} L \cong
  \ZZ^{r} \oplus G_{\text{tor}}$.  Projecting onto the torsion-free quotient,
  let $A'$ denote the image of $A$ in $\ZZ^{r}$.  In this setting, the
  saturation of $A$ is $\overline{A'} \oplus G_{\text{tor}}$.  With these
  definitions, one verifies that the integral closure of the semigroup ring
  coincides the semigroup ring of its saturation. Moreover, the monomials
  corresponding to $\eta(\ceil{\bm{p}}) \in \ZZ^n/L$, where $\bm{p}$ runs over
  the vertices of $\RL$, still generate the saturation.
\end{remark}

\section{Applications of the ceiling stratification}
\label{s:app}

\noindent
In this section, we use ceiling stratifications to describe resolutions of
sheaves on a smooth toric variety. We relate the cellular free resolutions in
\Cref{t:resSat} to the locally-free resolutions in \cite{HHL24}*{Theorem~A}
and the minimal free resolutions in \cite{BE24}*{Theorem~1.2}. In particular,
we show that, for a smooth toric variety, the resolutions of the diagonal from
\cite{HHL24} and \cite{BE24} coincide with the $\mathcal{O}_{X}$-complex
associated to a cellular free resolution.

Throughout, the morphism $\varphi \colon Y \to X$ is a torus-equivariant
embedding of a normal toric variety $Y$ into a smooth toric variety $X$ with
no torus factors. The polynomial ring $S \colequal \kk[x_1, x_2, \dotsc, x_n]$
is the Cox ring of $X$. Following \Cref{r:toric}, the lattice arising from the
map $\varphi$ is $L \colequal \Ker(\overline{\varphi}^{*}) \subseteq \ZZ^{n}$.

\subsection*{Recovering Hanlon--Hicks--Lazarev resolutions}

Given $\varphi \colon Y \to X$, \cite{HHL24}*{Theorem~3.5} produces a
combinatorial resolution of the sheaf $\varphi_{*} \mathcal{O}_{Y}$ in terms
of direct sums of line bundles on $X$.  As explained in \cite{HHL24}*{\S3.4},
the relevant chain complex of sheaves is obtained from a cell structure (also
referred to as topological `stratification') on the real torus corresponding
to an `exit path category.'  Moreover, \cite{HHL24}*{Proposition~3.9}
introduces a functor from this exit path category to the category of coherent
sheaves on $X$ that assigns a line bundle to each open cell.

The subsequent lemma aligns the ceiling stratification in \Cref{p:strat} with
the categorical construction in \cite{HHL24}*{\S3}. As in \Cref{r:toric}, we
identify $\ZZ^n$ with the group of torus-invariant divisors on $X$. For any
lattice point $\bm{u} \in \ZZ^n$, we write $\mathcal{O}_{X}(\bm{u})$ for the
corresponding line bundle.

\begin{lemma}
  \label{l:hhl}
  The ceiling stratification $\psi \colon \RL \to \ZZ^n$ and the construction
  in \cite{HHL24}*{\S3.4} agree: the induced cell structures on the real torus
  $\RL \mathbin{\!/\!} L$ coincide, and the line bundle assigned to the cell
  $\sigma' \subset \RL \mathbin{\!/\!} L$ is
  $\mathcal{O}_{X} \bigl( - \psi(\sigma) \kern-1.0pt \bigr)$, where the cell
  $\sigma \subset \RL$ represents the coset $\sigma'$.
\end{lemma}

\begin{proof}
  As in \Cref{r:toric}, the toric morphism $\varphi \colon Y \to X$
  corresponds to a surjective $\ZZ$-linear map
  $\overline{\varphi}^{*} \colon M_X \to M_Y$ between the character lattices
  of $X$ and $Y$, and $L \colequal \Ker(\overline{\varphi}^{*})$.  In
  \cite{HHL24}*{\S3.4}, the real torus is defined to be the kernel of the
  induced map
  $(M_X \otimes_{\ZZ} \RR) \mathbin{\!/\!} M_X \to (M_Y \otimes_{\ZZ} \RR)
  \mathbin{\!/\!} M_Y$ on tori.  Hence, this real torus is precisely the real
  topological torus $\RL \mathbin{\!/\!} L$.

  We next analyze the cell structures on $\RL \mathbin{\!/\!}  L$. Following
  \Cref{r:toric} or \Cref{p:strat}, the space $\RL$ carries the standard cell
  structure. The inclusion $\nu \colon M_X \to \ZZ^n$ sends a character
  $\bm{m} \in M_{X}$ to the lattice point in $\ZZ^n$ whose $i$th entry is
  $\bm{u}_{i} \cdot \bm{m}$, where $\bm{u}_{i}$ is the primitive lattice point
  generating the $i$th ray in the fan of $X$; see
  \cite{CLS11}*{Proposition~4.1.2}.  The composition of the canonical
  inclusion $\kappa \colon L \to M_X$ with $\nu$ realizes the lattice $L$ as a
  subgroup of $\ZZ^n$. By \Cref{d:std}, the standard cell structure on $\RL$
  is determined by restricting the periodic arrangement of the standard basis
  vectors in $\RR^n$. As explained in \Cref{r:cell}, this polyhedral cell
  structure on $\RL$ arises from the hyperplanes of the form
  $\set{ \bm{p} \in \RL \mathrel{\big|} \bm{u}_{i} \cdot \bm{p} = j}$, where
  $1 \leqslant i \leqslant n$ and $j$ is an integer.  Under the quotient map
  $\pi \colon \RL \to \RL \mathbin{\!/\!}  L$, this periodic arrangement
  descends to a toric arrangement on the real torus $\RL \mathbin{\!/\!}
  L$. On the other hand, the topological stratification in
  \cite{HHL24}*{\S3.4} is also induced by the toric arrangement
  $\RL \mathbin{\!/\!} L$ consisting of the subtori determined by the linear
  forms $\bm{u}_{i}$ for all $1 \leqslant i \leqslant n$. We deduce that the
  two induced cell structures on the real torus $\RL \mathbin{\!/\!}  L$ are
  identical.

  It remains to relate the line bundles in \cite{HHL24}*{Proposition~3.5} to
  the ceiling stratification $\psi \colon \RL \to \ZZ^n$. For each point
  $\bm{q} \in M_X \otimes_{\ZZ} \RR$, \cite{HHL24}*{Equation~(6) in \S2.3}
  assigns the line bundle
  \[
    \mathcal{O}_{X} \bigl( (\floor{ - \bm{u}_1 \cdot \bm{q}}, \floor{
      -\bm{u}_2 \cdot \bm{q}}, \dotsc, \floor{ - \bm{u}_n \cdot \bm{q}})
    \bigr)
    = \mathcal{O}_{X} \bigl( - \ceil{\nu_{\RR}(\bm{q})} \bigr) \, ,
  \]
  where $\nu_{\RR} \colon M_{X} \otimes_{\ZZ} \RR \to \RR^n$ is the linear
  extension of $\nu \colon M_X \to \ZZ^n$. Since
  $\iota \colequal \nu \circ \kappa$ realizes the lattice $L$ as a subgroup of
  $\ZZ^n$, its linear extension $\iota_{\RR}$ realizes $\RL$ as a linear
  subspace of $\RR^n$. By construction, the ceiling function is constant on
  the open cells in $\RL$, so the line bundle assigned to every point in
  $\sigma$ is $\mathcal{O}_{X} \bigl( - \psi(\sigma) \kern-1.0pt
  \bigr)$. Passing to the real torus $\RL \mathbin{\!/\!}  L$, we see that the
  ceiling stratification and the functor in \cite{HHL24}*{Proposition~3.5}
  assign the same line bundle to each open cell.
\end{proof}

With the relevant building blocks in place, we compare resolutions of the
sheaf $\varphi_{*} \mathcal{O}_{Y}$.  Following \cite{CLS11}*{\S5.3}, each
finitely generated $S$-module is associated to a quasi-coherent
$\mathcal{O}_{X}$-module.

\begin{theorem}
  \label{t:hhl}
  Let $\varphi \colon Y \to X$ be a torus-equivariant embedding of a normal
  toric variety $Y$ into a smooth toric variety $X$ with no torus factors, and
  let $L \subseteq \ZZ^{n}$ be the lattice arising from $\varphi$. For the
  ceiling stratification $\psi \colon \RL \to \ZZ^n$, the
  $\mathcal{O}_{X}$-complex associated to the $S$-complex
  $F_{\psi} \otimes_{S[L]} S$ is isomorphic to the resolution of the
  $\mathcal{O}_{X}$-module $\varphi_{*} \mathcal{O}_{Y}$ in
  \cite{HHL24}*{Theorem~3.5}.
\end{theorem}

\begin{proof}
  For any lattice point $\bm{u} \in \ZZ^n$, the coherent sheaf associated to
  the $\ZZ^{n}$-graded $S$-module $S(\bm{u})$ is the torus-equivariant line
  bundle $\mathcal{O}_{X}(\bm{u})$; see \cite{CLS11}*{Proposition~5.3.7}. From
  the description of the $S$-complex $F_{\psi} \otimes_{S[L]} S$ in
  \Cref{r:com}, it follows that the $i$th term in its associated
  $\mathcal{O}_{X}$-complex is
  \[
    \bigoplus_{\begin{subarray}{c}
        \sigma' \subset \RL \mathbin{/} L \\
        \dim \sigma' = i
      \end{subarray}} \mathcal{O}_{X} \bigl( - \psi(\sigma) \kern-1.0pt \bigr),
  \]
  where $\sigma$ is a representative in $\RL$ for the coset $\sigma'$
  in the quotient $\RL \mathbin{\!/\!}  L$.  Combining \Cref{l:hhl} with
  \cite{HHL24}*{Equation~(18) in \S{A.1}} shows that this direct sum is also
  the $i$th term in the resolution in \cite{HHL24}*{Theorem~3.5}.

  We next examine the differentials. Following \Cref{d:comp}, the monomials
  $\bm{x}^{\psi(\sigma) - \psi(\tau)} \in S$ appearing in the differentials
  determine the unique homomorphism from
  $S \bigl( -\psi(\sigma) \kern-1.0pt \bigr)$ to
  $S \bigl( -\psi(\tau) \kern-1.0pt \bigr)$ having degree $\bm{0} \in \ZZ^n$.
  Moreover, monomials in the Cox ring $S$ are identified with global sections
  of line bundles; see \cite{CLS11}*{Proposition~5.3.7}.  Hence, the
  associated-sheaf functor sends the matrix
  $[\bm{x}^{\psi(\sigma) - \psi(\tau)}]$ to the unique torus-equivariant
  morphism from $\mathcal{O}_{X} \bigl( -\psi(\sigma) \kern-1.0pt \bigr)$ to
  $\mathcal{O}_{X} \bigl( -\psi(\tau) \kern-1.0pt \bigr)$. On the other hand,
  \cite{HHL24}*{Proposition~2.14} identifies the components of the
  differentials as the unique degree-$\bm{0}$ maps between relevant summands
  of the direct sum. Therefore, the associated-sheaf functor applied to the
  differentials in \Cref{r:com} agrees, up to scalars, with the differentials
  in \cite{HHL24}*{Equation~(19) in \S{A.1}}.

  Finally, \Cref{r:com} shows that the scalars in the differentials of the
  $S$-complex $F_{\psi} \otimes_{S[L]} S$ are determined by the incidence
  function $\varepsilon$.  The construction in \cite{HHL24}*{\S3.5} allows for
  an arbitrary choice of `orientation.'  As observed in the paragraph
  following \cite{HHL24}*{Example~A.3}, different choices of incidence
  functions produce isomorphic $\mathcal{O}_{X}$-complexes.
\end{proof}

\begin{remark}
  Since the construction of the $S$-complex $F_{\psi} \otimes_{S[L]} S$
  applies in all characteristics, Theorem~\ref{t:hhl} implies that the
  construction of the $\mathcal{O}_{X}$-complex in \cite{HHL24}*{\S3.4} is also
  independent of the characteristic of the underlying field $\kk$; compare
  with \cite{B+24}*{Remark~6.2}.
\end{remark}

\begin{remark}
  Unlike \Cref{t:hhl}, \cite{HHL24}*{Theorem~3.5} relaxes the hypotheses on
  the ambient space $X$, allowing smooth toric stacks.  For brevity, we again
  leave the variant of \Cref{t:hhl} for smooth toric stacks to the interested
  reader.
\end{remark}

Reversing direction, the $S$-complex associated to the resolution in
\cite{HHL24}*{Theorem~3.5} yields a cellular free resolution.  Following
\cite{CLS11}*{Definition~6.A.1}, set
\[
  \Gamma_{*}(\mathcal{F})
  \colequal \bigoplus_{\bm{u} \in \Pic(X)} H^{0} \bigl( X, \mathcal{F}(\bm{u})
  \kern-1.0pt \bigr)
  = \bigoplus_{\bm{u} \in \Pic(X)} H^{0} \bigl( X, \mathcal{F}
  \otimes_{\mathcal{O}_{X}} \mathcal{O}_{X}(\bm{u}) \kern-1.0pt \bigr)
\]
for any sheaf $\mathcal{F}$ of $\mathcal{O}_{X}$-modules on $X$.

\begin{corollary}
  The functor $\Gamma_{*}$ applied to the resolution of
  $\varphi_{*} \mathcal{O}_{Y}$ in \cite{HHL24}*{Theorem~3.5} is isomorphic to
  the $S$-complex $F_{\psi} \otimes_{S[L]} S$ associated to the ceiling
  stratification $\psi \colon \RL \to \ZZ^n$.
\end{corollary}

\begin{proof}[Sketch of proof]
  To capitalize on the torus action, assume that
  $S \colequal \kk[x_1, x_2, \dotsc, x_n]$ has the canonical $\ZZ^{n}$-grading
  and that all line bundles on $X$ have a chosen torus-equivariant structure.
  For any lattice point $\bm{u} \in \ZZ^{n}$, \cite{CLS11}*{Proposition~5.3.7}
  demonstrates that
  $\Gamma_{*} \bigl( \mathcal{O}_{X}(\bm{u}) \kern-1.0pt \bigr) = S(\bm{u})$,
  which implies that the terms of the complexes are identical. Since
  \[
    \cHom \bigl( \mathcal{O}_{X}(\bm{u}), \mathcal{O}_{X}(\bm{v}) \kern-1.0pt
    \bigr)
    \cong\, \cHom \bigl( \mathcal{O}_{X}, \mathcal{O}_{X}(\bm{v} - \bm{u})
    \kern-1.0pt \bigr)
    \cong \mathcal{O}_{X}(\bm{v}- \bm{u}) \, ,
  \]
  torus-equivariant maps from $\mathcal{O}_{X}(\bm{u})$ to
  $\mathcal{O}_{X}(\bm{v})$ are the same as monomial maps from $S(\bm{u})$ to
  $S(\bm{v})$, which are unique up to scaling. Thus, the differentials agree
  up to a choice of incidence function.
\end{proof}

\subsection*{Relation to Brown--Erman resolutions}

Addressing a challenge raised in \cite{BE24}*{p.~3}, cellular free resolutions
allow a comparison between the minimal free resolutions in
\cite{BE24}*{Theorem~1.2} and the locally-free resolutions in
\cite{HHL24}*{Theorem~3.5}. The next corollary shows that a resolution from
\cite{BE24}*{Theorem~1.2} is always isomorphic to a direct summand of one from
\cite{HHL24}*{Theorem~3.5}.

\begin{corollary}
  \label{c:hhl}
  Let $\varphi \colon Y \to X$ be a torus-equivariant embedding of a normal
  toric variety $Y$ into a smooth toric variety $X$ with no torus factors.
  When $F$ is a minimal free resolution of the integral closure of the
  $S$-module $S \mathbin{\!/\!} I_L$, the $\mathcal{O}_{X}$-complex associated
  to $F$ is isomorphic to a direct summand of the resolution in
  \cite{HHL24}*{Theorem~3.5}.
\end{corollary}

\begin{proof}
  Let $L$ be the sublattice of $\ZZ^{n}$ in \Cref{r:toric}, and let
  $\psi \colon L_\RR \to \ZZ^{n}$ denote the ceiling stratification.
  \Cref{t:resSat} establishes that the $S$-complex $F_{\psi} \otimes_{S[L]} S$
  is a free resolution of the $S$-module
  $\smash{\overline{S \mathbin{\!/\!} I_L}}$.  Being a minimal free resolution
  of $\smash{\overline{S \mathbin{\!/\!} I_L}}$, it follows that $F$ is
  isomorphic to a direct summand of $F_{\psi} \otimes_{S[L]} S$. \Cref{t:hhl}
  shows that the $\mathcal{O}_{X}$-complex associated to
  $F_{\psi} \otimes_{S[L]} S$ is isomorphic to the resolution of the
  $\mathcal{O}_{X}$-module $\varphi_{*} \mathcal{O}_{Y}$ in
  \cite{HHL24}*{Theorem~3.5}.  Since the associated-sheaf functor commutes with
  finite direct sums, we conclude that the $\mathcal{O}_{X}$-complex
  associated to $F$ is isomorphic to a direct summand of the resolution in
  \cite{HHL24}*{Theorem~3.5}.
\end{proof}

We also determine when a cellular free resolution is minimal.  The criterion
for minimality applies to every compatible stratification and is a minor
variant of \cite{BS98}*{Remark~1.4}.

\begin{lemma}
  \label{l:cond}
  Let $\psi \colon L \to \ZZ^{n}$ be a compatible stratification.  The
  cellular free $S$-complex $F_{\psi}$ is minimal if and only if, for every
  lattice point $\bm{u} \in \ZZ^n$, each connected component of the preimage
  $\psi^{-1}(\bm{u})$ consists of a single open cell.
\end{lemma}

\begin{proof}
  A free $S$-complex is minimal if and only if no entry in the matrices of the
  differentials is a unit.  From \Cref{d:comp}, it follows that $F_{\psi}$ is
  minimal if and only if, for all open cells $\sigma, \tau \subset \RL$ such
  that $\tau$ lies in the closure of $\sigma$ and
  $\dim \sigma = 1 + \dim \tau$, we have $\psi(\sigma) \neq \psi(\tau)$.

  Fix a lattice point $\bm{u} \in \ZZ^n$.  Suppose that the open cell $\sigma$
  is contained in the preimage $\psi^{-1}(\bm{u})$.  When $F_{\psi}$ is
  minimal, none of the open cells $\tau$ satisfying
  $\tau \subseteq \overline{\sigma}$ and $\dim \sigma = 1 + \dim \tau$ are
  contained in $\psi^{-1}(\bm{u})$.  Hence, the connected component of this
  preimage that contains $\sigma$ contains no other cells.  Conversely, if
  each connected component of the preimage $\psi^{-1}(\bm{u})$ is a single
  open cell, then no open cell in the closure has the same value under $\psi$.
\end{proof}

\begin{remark}
  \label{r:min}
  When each preimage of $\psi$ lies in the closure of a single open cell, the
  criterion in \Cref{l:cond} reduces to verifying that the preimages are
  relatively open.
\end{remark}

One important family of embeddings always produces minimal free resolutions.
Assume that $Y$ is a smooth toric variety with no torus factors and let $m$ be
the number of rays in its fan.  The product $X \colequal Y \times Y$ is also a
smooth toric variety with no torus factors and $n \colequal 2m$ rays in its
fan. As in \Cref{e:diag}, consider the diagonal embedding
$\Delta \colon Y \to X$, where the lattice is
\[
L \colequal \Lambda(M_{Y}) \colequal \smash{\set{(\bm{v}, - \bm{v}) \in \ZZ^{m}
    \times \ZZ^{m} \mathrel{|} \bm{v} \in M_Y}} \subset \ZZ^{n} \, .
\]
The ceiling stratification $\psi \colon \RL \to \ZZ^{n}$ is defined by $\psi
\bigl( \bm{q},-\bm{q} \bigr) = (\ceil{\bm{q}}, \ceil{-
  \bm{q}}) = (\ceil{\bm{q}}, - \floor{\bm{q}})$ for all
$\bm{q} \in (M_{Y} \otimes_{\ZZ} \RR)\subseteq \RR^{m}$.

\begin{theorem}
  \label{t:cell-min}
  Let $Y$ be a smooth toric variety with no torus factors, and let $S$ denote
  the Cox ring of the product $Y \times Y$.  When $\psi$ is the ceiling
  stratification for the diagonal embedding $\Delta \colon Y \to Y \times Y$,
  the cellular $S$-complex $F_{\psi} \otimes_{S[L]} S$ is a minimal free
  resolution.
\end{theorem}

\begin{proof}
  Since \Cref{t:resSat} establishes that this cellular free $S$-complex is a
  resolution of the integral closure of the $S$-module
  $S \mathbin{\!/\!} I_L$, it suffices to show that
  $F_{\psi} \otimes_{S[L]} S$ is minimal.

  We analyze the preimage $\smash{\psi^{-1}\bigl( \kern-1.0pt (\bm{u},\bm{v})
    \kern-1.0pt \bigr)}$, where $(\bm{u}, \bm{v}) \in \ZZ^{2m}$. Observe that
  $\psi \bigl( \bm{q}, -\bm{q} \bigr) = (\bm{u}, \bm{v})$ if and only if
  $u_{i} = \ceil{q_{i}}$ and $v_{i} = \ceil{-q_{i}}$ for all $1 \leqslant i
  \leqslant m$. Since
  \[
    \ceil{q_{i}} + \ceil{-q_i}
    = \begin{cases}
      0 & \text{if $q_{i} \in \ZZ$} \\[-2pt]
      1 & \text{if $q_{i} \not\in \ZZ$,} \\
    \end{cases}
  \]
  it follows that either $u_{i} + v_{i} = 0$, $u_{i} + v_{i} = 1$, or
  $\smash{\psi^{-1}\bigl( \kern-1.0pt (\bm{u},\bm{v}) \kern-1.0pt \bigr)}$ is
  empty.  Fixing the index $i$, the first case corresponds to the subset
  $\smash{\set{\bm{p} \in \RL \mathrel{\big|} q_i = u_{i} }}$ and the second
  corresponds to the subset $\smash{\set{\bm{p} \in \RL \mathrel{\big|} u_{i}
      -1 < q_i < u_{i} }}$.  Since each nonempty $\smash{\psi^{-1}\bigl(
    \kern-1.0pt (\bm{u},\bm{v}) \kern-1.0pt \bigr)}$ is the finite
  intersection of such subsets, it is relatively open. Since no hyperplane
  from the arrangement of the standard cell structure intersects the relative
  interior of $\smash{\psi^{-1}\bigl( \kern-1.0pt (\bm{u},\bm{v}) \kern-1.0pt
    \bigr)}$, the preimage lies in the closure of a single cell.

  Combining \Cref{l:cond} and \Cref{r:min} establishes the minimality of the
  $S[L]$-complex $F_{\psi}$.  From the description of the differentials in
  \Cref{r:com}, we deduce that $F_{\psi} \otimes_{S[L]} S$ is also minimal.
\end{proof}

\begin{corollary}
  \label{c:HHL-BE}
  Let $Y$ be a smooth toric variety with no torus factors, and let $S$ be the
  Cox ring of $X \colequal Y \times Y$.  When $\Delta \colon Y \to X$ is the
  diagonal embedding, the resolutions of the $\mathcal{O}_{X}$-module
  $\Delta_{*}^{} \mathcal{O}_{Y}$ in \cite{HHL24}*{Theorem~3.5} and
  \cite{BE24}*{Theorem~1.2} are both isomorphic to the $\mathcal{O}_{X}$-complex
  associated to the minimal cellular free resolution of the integral closure
  of the $S$-module $S \mathbin{\!/\!} I_L$.
\end{corollary}

\begin{proof}
  As in \Cref{r:toric}, let $L \subseteq \RR^{n}$ be the lattice arising from
  $\Delta$, and let $\psi \colon \RL \to \ZZ^{n}$ be the ceiling
  stratification.  By \Cref{t:hhl}, the resolution of the
  $\mathcal{O}_{X}$-module $\Delta_{*}^{} \mathcal{O}_{Y}$ in
  \cite{HHL24}*{Theorem~3.5} is isomorphic to the $\mathcal{O}_{X}$-complex
  associated to the $S$-complex $F_{\psi} \otimes_{S[L]} S$.  \Cref{t:resSat}
  and \Cref{t:cell-min} show that $F_{\psi} \otimes_{S[L]} S$ is a minimal
  free resolution of the integral closure of the $S$-module
  $S \mathbin{\!/\!} I_L$. The resolution of $\Delta_{*}^{} \mathcal{O}_{Y}$
  in \cite{BE24}*{Theorem~1.2} is the $\mathcal{O}_{X}$-module associated to
  the minimal free resolution of the integral closure of the $S$-module
  $S \mathbin{\!/\!} I_L$, which completes the proof.
\end{proof}

\section{Applications to other resolutions of the diagonal}
\label{s:oth}

\noindent
By finding suitable compatible $\ZZ^n$-stratifications, this section
reinterprets the constructions from \cite{And24}*{Definition~3.4} and
\cite{FH25}*{Example~3.15} via cellular free resolutions.

\subsection*{Sheaves arising from nonceiling stratifications}

We first demonstrate that \Cref{p:nor} does not extend to all compatible
$\ZZ^{m}$-stratifications.  Precomposing a ceiling stratification with a
translation by a real vector produces the desired counterexample.

\begin{example}[A module derived from a nonceiling stratifcation]
  \label{e:nonEx}
  Let $S = \kk[x_1, x_2, y_1, y_2]$ be the Cox ring of $\PP^1 \times
  \PP^1$. The lattice $L \subset \ZZ^4$ generated by $(1,-1,-1,1)$ determines
  the toric ideal $I_{L} = \ideal{x_1 y_2 - x_2 y_1}$, and this ideal cuts out
  the image of the diagonal embedding $\Delta \colon \PP^1 \to \PP^1 \times
  \PP^1$; compare with \Cref{e:diag}. Endow $\RL \cong \RR^1$ with an
  $\ZZ$-equivariant cell complex such that the fundamental domain has three
  vertices and the two $1$-cells corresponding to the points $0$, $0.5$, and
  $1$, and the open intervals $(0,0.5)$ and $(0.5,1)$ in $\RR^1$. Let
  $\iota_\RR \colon \RR^1 \to \RR^4$ be the inclusion with image $\RL$.

  One verifies that the map
  $\psi \colon \RL \to \ZZ^4$ defined by
  $\psi \bigl( \iota_{\RR} (q) \kern-1.0pt \bigr) \colequal \ceil{(q-0.5, -q,
    -q-0.5, q)}$ for all $q \in \RR^1$ is a compatible
  $\ZZ^4$-stratification.  \Cref{p:gens} shows that the $S$-module
  $\smash{M_{\psi} \otimes_{S[L]} S}$ is generated by the monomials $1$ and
  $\smash{y_1^{-1} y_{2}^{}}$ in the group algebra
  $\kk[\ZZ^{4} \mathbin{\!/\!} L]$.  It follows that the sheaf on
  $\smash{\PP^1 \times \PP^1}$ associated to
  $\smash{M_{\psi} \otimes_{S[L]} S}$ coincides with the one associated to
  $\smash{\bigl(S \mathbin{\!/\!} I_L \bigr) (0,1)}$.  The latter is not the
  sheaf associated to $S \mathbin{\!/\!} I_L$, which is the same as the sheaf
  associated to $\smash{\overline{S \mathbin{\!/\!} I_L}}$.
\end{example}

Despite this example, the $S$-modules $M_\psi \otimes_{S[L]} S$ and
$S \mathbin{\!/\!}  I_L$ retain a geometric connection. For the remainder of
this subsection, let $\varphi \colon Y \to X$ be a torus-equivariant embedding
of a normal toric variety $Y$ into a smooth toric variety $X$ with no torus
factors. The Cox ring of $X$ is the polynomial ring
$S \colequal \kk[x_1, x_2, \dotsc, x_n]$. As in \Cref{r:toric},
$L \subseteq \ZZ^n$ is the lattice arising from $\varphi$.

\begin{lemma}
  \label{l:supp}
  For any compatible $\ZZ^n$-stratification $\psi \colon \RL \to \ZZ^n$, the
  $\mathcal{O}_{X}$-modules associated to the $S$-modules
  $M_\psi \otimes_{S[L]} S$ and $S \mathbin{\!/\!} I_L$ coincide on the dense
  algebraic torus in $X$.  Moreover, the support of the sheaf associated to
  $M_\psi \otimes_{S[L]} S$ equals the subvariety in $X$ cut out by the toric
  ideal $I_{L}$.
\end{lemma}

\begin{proof}
  Since
  $R \colequal \kk[x_1^{\pm 1}, x_{2}^{\pm 1}, \dotsc, x_{n}^{\pm 1}] =
  \kk[\ZZ^{n}]$, \cite{CLS11}*{Proposition~3.3.11 and Proposition~5.3.3} show
  that the coordinate ring of the dense algebraic torus in $X$ is
  $( S_{x_1x_2 \dotsc x_n})_{\bm{0}} = (R)_{\bm{0}}$.  Hence, it suffices to
  prove that the modules $(M_\psi \otimes_{S[L]} S) \otimes_{S} R$ and
  $(S \mathbin{\!/\!} I_L) \otimes_{S} R$ are isomorphic. Both
  $M_\psi \otimes_{S[L]} S$ and $S \mathbin{\!/\!} I_L$ are $S$-submodules of
  $\kk[\ZZ^n \mathbin{\!/\!} L]$ containing a monomial.  As any monomial
  generates $\kk[\ZZ^n \mathbin{\!/\!} L]$ as an $R$-module, tensoring with
  $R$ yields the required isomorphisms
  \[
    (M_\psi \otimes_{S[L]} S) \otimes_{S} R
    \cong \kk[\ZZ^n \mathbin{\!/\!} L]
    \cong (S \mathbin{\!/\!} I_L) \otimes_{S} R \, .
  \]

  Because every nonzero submodule of $\kk[\ZZ^n \mathbin{\!/\!} L]$ intersects
  $M_\psi \otimes_{S[L]} S$ nontrivially, the annihilator of the $S$-module
  $M_\psi \otimes_{S[L]} S$ equals the annihilator of the $S$-module
  $\kk[\ZZ^n \mathbin{\!/\!} L]$, which is the toric ideal $I_{L}$.  Thus, the
  second assertion follows.
\end{proof}

We also provide a sufficient condition to ensure that a compatible
$\ZZ^{n}$-stratification $\psi \colon \RL \to \ZZ^{n}$ gives rise to a
locally-free resolution of the sheaf associated to the $S$-module $S
\mathbin{\!/\!} I_L$; compare with \cite{CLS11}*{Proposition~5.3.10}.  As in
\cite{BS98}*{\S0}, the $S$-module arising from $L$ is
\[
  M_{L}
  \colequal S \cdot \set{x^{\bm{v}} \mathrel{\big|} \bm{v} \in L} \subset R
  \, .
\]
For any cone $\sigma$ in the fan $\Sigma_X$ of the toric variety $X$, the
monomial $\smash{\bm{x}^{\widehat{\sigma}}}$ is defined as the product of the
variables in $S$ corresponding to the rays not in $\sigma$.

\begin{proposition}
  \label{p:sheaf}
  Let $\psi \colon \RL \to \ZZ^{n}$ be a compatible $\ZZ^{n}$-stratification
  such that $\psi(\bm{0}) = \bm{0}$ and the cellular free $S[L]$-complex
  $F_{\psi}$ is a resolution.  The $\mathcal{O}_{X}$-modules associated to the
  $S$-modules $M_\psi \otimes_{S[L]} S$ and $S \mathbin{\!/\!}  I_L$ coincide
  if, for all $\sigma \in \Sigma_{X}$, there exists a nonnegative integer $k$
  such that
  \[
    (\bm{x}^{\widehat{\sigma}})^{k} \cdot M_{\psi} \subseteq M_{L} \, .
  \]
\end{proposition}

\begin{proof}
  Fix a cone $\sigma \in \Sigma_{X}$. As $\psi(\bm{0}) = \bm{0}$, there is an
  inclusion $M_{L} \subseteq M_{\psi}$ of $S[L]$\nobreakdash-modules.  By
  hypothesis, there is a nonnegative integer $k$ such that
  $\smash{(\bm{x}^{\widehat{\sigma}})^{k}} \cdot M_{\psi} \subseteq M_{L}$.
  Combining these inclusions gives
  $(M_{L})_{\bm{x}^{\widehat{\sigma}}} =
  (M_{\psi})_{\bm{x}^{\widehat{\sigma}}}$.  Since \cite{BS98}*{Lemma~3.1}
  establishes that $M_{L} \otimes_{S[L]} S = S \mathbin{\!/\!}  I_L$, we
  obtain
  \begin{align*}
    \bigl( (S \mathbin{\!/\!}  I_L)_{\bm{x}^{\widehat{\sigma}}} \bigr)_{\bm{0}}
    = \bigl( (M_{L} \otimes_{S[L]} S)_{\bm{x}^{\widehat{\sigma}}} \bigr)_{\bm{0}}
    &= \bigl( (M_{L})_{\bm{x}^{\widehat{\sigma}}} \otimes_{S[L]}
    (S)_{\bm{x}^{\widehat{\sigma}}} \bigr)_{\bm{0}} \\
    &= \bigl( (M_{\psi})_{\bm{x}^{\widehat{\sigma}}} \otimes_{S[L]}
    (S)_{\bm{x}^{\widehat{\sigma}}} \bigr)_{\bm{0}}
    = \bigl( (M_{\psi} \otimes_{S[L]} S)_{\bm{x}^{\widehat{\sigma}}}
    \bigr)_{\bm{0}} \, .
  \end{align*}
  Applying \cite{CLS11}*{Proposition~5.3.3}, the sheaves associated
  to the $S$-modules $M_\psi \otimes_{S[L]} S$ and $S \mathbin{\!/\!}  I_L$
  are identical on an open affine cover of $X$. Therefore, these
  $\mathcal{O}_{X}$-modules coincide.
\end{proof}

\begin{remark}
  \label{r:irr}
  One can strengthen \Cref{p:sheaf} by reducing the set of cones that need to
  be considered. To see this, recall from \Cref{r:toric} that the toric
  morphism $\varphi \colon Y \to X$ corresponds to an injective $\ZZ$-linear
  map $\overline{\varphi} \colon N_{Y} \to N_{X}$ between the lattices of
  one-parameter subgroups of $Y$ and $X$. By \Cref{l:supp}, the support of the
  sheaf associated to $M_\psi \otimes_{S[L]} S$ is the subvariety of $X$ cut
  out by the toric ideal $I_{L}$. Thus, it suffices to check that these
  sheaves agree on the affine opens $U_{\sigma}$ covering $\varphi(Y)$.  One
  must verify that there exists a nonnegative integer $k$ such that
  $(\bm{x}^{\widehat{\sigma}})^{k} \cdot M_{\psi} \subseteq M_{L}$ for all
  cones $\sigma$ satisfying $\dim \bigl( \overline{\varphi}(N_Y) \cap \sigma
  \bigr) = \dim N_{Y}$.
\end{remark}

\subsection*{Revisiting the Anderson resolution of the diagonal}

We place the complex constructed by Anderson in \cite{And24}*{Definition~3.4} within
our framework. Let $Y$ be a smooth toric variety with no torus factors, and
let $m$ be the number of rays in its fan. The diagonal embedding $\Delta
\colon Y \to X \colequal Y \times Y$ gives rise to the lattice $L =
\Lambda(M_{Y}) \colequal \smash{\set{(\bm{v}, - \bm{v}) \in \ZZ^{m} \times
    \ZZ^{m} \mathrel{|} \bm{v} \in M_Y \subset \ZZ^m}} \subset \ZZ^{n}$, see
also \Cref{e:diag}. The toric ideal $I_{L}$ cuts out the image of the diagonal
embedding. Inspired by \cite{And24}*{\S3}, we christen the following
compatible $\ZZ^{n}$-stratification; compare with \Cref{e:lcm}. We continue to
assume that the space $\RL \cong \RR^{m}$ carries the standard cell structure.

\begin{definition}
  \label{d:And}
  For any Lawrence lattice $L = \Lambda(M_{Y}) \subset \ZZ^n$, the
  \emph{Anderson stratification} is the map $\psi \colon \RL \to \ZZ^{n}$
  defined by $\psi(\bm{q}, -\bm{q}) = (\floor{\bm{q}}, - \floor{\bm{q}})$ for
  each vertex $(\bm{q}, -\bm{q}) \in \RL \subset \RR^n$.  For any
  higher-dimensional open cell $\sigma$, the value $\psi(\sigma)$ is the
  componentwise maximum of the lattice vectors assigned to the vertices in the
  closure of $\sigma$.
\end{definition}

In general, the Anderson stratification $\psi \colon \RL \to \ZZ^{n}$ does not
satisfy the topological condition in \Cref{p:res}, and the cellular free
$S$-complex $F_\psi \otimes_{S[L]} S$ may fail to be a resolution.  The
statement of \cite{And24}*{Theorem~1.1}
includes a central symmetry hypothesis, which we do not require, and does not
address exactness at the level of modules.  As the next example illustrates,
even when the complex is a resolution, it need not resolve the $S$-module
$M_\psi \otimes_{S[L]} S$.

\begin{example}[Cellular free resolution arising from an Anderson
  stratification]
  \label{e:And}
  As in \Cref{e:diag}, consider the toric surface
  $Y \colequal \smash{\PP \bigl(\mathcal{O}_{\PP^1} \oplus
    \mathcal{O}_{\PP^1}(2) \!\bigr)}$ with its diagonal embedding
  $\Delta \colon Y \to X \colequal Y \times Y$, and write
  $S \colequal \kk[x_1, x_2, x_3, x_4, y_1, y_2, y_3, y_4]$ for the Cox ring of
  $X$. Let $\psi \colon \RL \to \ZZ^{8}$ denote the Anderson
  stratification. Figure~\ref{f:And24} displays a fundamental domain in which
  each vertex $\bm{q} \in \RL$ is labeled with the monomial having exponent
  vector $\psi(\bm{q})$.  The corresponding $\Pic(X)$-graded cellular free
  $S$-complex $F_\psi \otimes_{S[L]} S$ is
  \[
    \begin{matrix}
      S(0, 0, \; 0, 0) \\[-3pt]
      \oplus \\[-3pt]
      S(-1, 1, \; 1, -1) \\
    \end{matrix}
    \xleftarrow{\quad \partial_1 \quad}
    \begin{matrix}
      \relphantom{-} S(0, 0, \; 1, -1)^2 \\[-3pt]  
      \oplus \\[-3pt]
      S(-1, 0, \; 0, -1)^2\\[-3pt] 
      \oplus \\[-3pt]
      S(-1, 0, \; -1, 0) \\  
    \end{matrix}
    \xleftarrow{\quad \partial_2 \quad} \!\!\!\!\!\!
    \begin{matrix}
      S(-1, 0, \; 1, -1) \\[-3pt]
      \oplus \\[-3pt]
      S(0, -1, \; 0, -1) \\[-3pt]
      \oplus \\[-3pt]
      \relphantom{-} S(-1, 0, \; -1, -1) \\
    \end{matrix} \!\!\!
    \xleftarrow{\qquad} 0
    \, ,
  \]
  where%
  \vspace{-1em}
  \begin{align*}
    \partial_1
    &=
      \begin{bmatrix*}
         y_2 y_3 &  y_1 y_2 & -x_3 y_4 & -x_1 y_4 & y_1 x_3 - x_1 y_3 \\
        -x_2 x_3 & -x_1 x_2 &  y_3 x_4 &  y_1 x_4 & 0 \\
      \end{bmatrix*}
    && \text{and}
    & \partial_2
    &= \begin{bmatrix*}
      -x_1 &  y_1 x_4 &  0   \\
       x_3 & -y_3 x_4 &  0   \\
       0   & -x_1 x_2 &  y_1 \\
       0   &  x_2 x_3 & -y_3 \\
      -y_2 &  0       &  y_4 \\
    \end{bmatrix*}
    \, .
  \end{align*}

  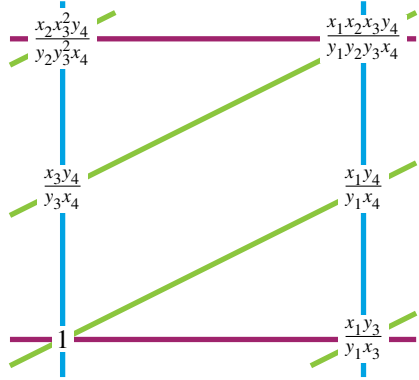
\begin{figure}[t]
    \centering
    \begin{tikzpicture}[scale=1.25,line width=2pt]
\usetikzlibrary{calc, decorations.pathreplacing,shapes.misc, arrows.meta}
\usetikzlibrary{decorations.pathmorphing}

\coordinate  (origin) at (0,0);
\coordinate (br) at (4,0);
\coordinate (tr) at (4,4);
\coordinate (tl) at (0,4);
\coordinate (ml) at (0,2);
\coordinate (mr) at (4,2);

\coordinate (hoff) at (0.7,0);
\draw[RedViolet] ($(origin)-(hoff)$) -- ($(br)+(hoff)$);
\draw[RedViolet] ($(tl)-(hoff)$)--($(tr)+(hoff)$);

\coordinate (voff) at (0,0.5);
\draw[Cerulean] ($(origin)-(voff)$) -- ($(tl)+(voff)$);
\draw[Cerulean] ($(br)-(voff)$) -- ($(tr)+(voff)$);

\coordinate (diagoff) at (0.7,0.35);
\draw[LimeGreen] ($(br)-(diagoff)$) -- ($(br)+(diagoff)$);
\draw[LimeGreen] ($(origin)-(diagoff)$) -- ($(mr)+(diagoff)$);
\draw[LimeGreen] ($(ml)-(diagoff)$) -- ($(tr)+(diagoff)$);
\draw[LimeGreen] ($(tl)-(diagoff)$) -- ($(tl)+(diagoff)$);

\node[inner sep = 1.5pt, fill=white] at (origin) {\footnotesize $1$};
\node[inner sep = 1.5pt, fill=white] (ml-label) at (ml)
{\footnotesize $\frac{x_3^{} y_4^{}}{y_3^{} x_4^{}}$};
\node[inner sep = 1.5pt, fill=white] at (tl)
{\footnotesize $\frac{x_2^{} x_3^{2} y_4^{}}{y_2^{} y_3^2 x_4^{}}$};
\node[inner sep = 1.5pt, fill=white] at (br)
{\footnotesize $\frac{x_1^{} y_3^{}}{y_1^{} x_3^{}}$};
\node[inner sep = 1.5pt, fill=white] (mr-label) at (mr)
{\footnotesize $\frac{x_1^{} y_4^{}}{y_1^{} x_4^{}}$};
\node[inner sep = 1.5pt, fill=white] at (tr)
{\footnotesize $\frac{x_1^{} x_2^{} x_3^{} y_4^{}}{y_1^{} y_2^{} y_3^{} x_4^{}}$};

\end{tikzpicture}
    \vspace{-0.75em}
    \caption{A fundamental domain of the Anderson stratification
      for the diagonal embedding of the second Hirzebruch surface.}
    \label{f:And24}
    \vspace{-0.5em}
  \end{figure}

  Lifting up to the universal cover, the cellular free $S$-complex $F_\psi$ is
  a resolution, but it does not resolve the monomial module $M_\psi$.  Observe
  that the vertices
  $\bm{p}_{1} \colequal (0, 0.5, 1, -0.5, \; 0, -0.5, -1, 0.5)$ and
  $\bm{p}_{2} \colequal (1, 0.5, 0, -0.5, \; -1, -0.5, 0, 0.5)$ in $\RL$
  satisfy $\psi(\bm{p}_{1}) = (0, 0, 1, -1, \; 0, 0, -1, 1)$ and
  $\psi(\bm{p}_{2}) = (1, 0, 0, -1, \; -1, 0, 0, 1)$. Since
  $x_1^{} x_3^{} x_{4}^{-1} y_{4}^{} = \operatorname{lcm}(x_3^{} y_3^{-1}
  x_{4}^{-1} y_{4}^{}, x_1^{} y_1^{-1} x_{4}^{-1} y_{4}^{})$, the failure to
  resolve $M_{\psi}$ in terms of \Cref{p:res} is witnessed by the closed
  subset $(\RL)_{\leqslant (1,0,1,-1, \; 0,0,0,1)}$, which consists precisely
  of these two vertices and is thereby not contractible.

  Alternatively, one certifies that the $S$-modules
  $H_{0}(F_\psi \otimes_{S[L]} S)$ and $M_\psi \otimes_{S[L]} S$ are different
  by showing that their annihilators are not equal.  As in the proof of
  \Cref{l:supp}, the annihilator of $M_\psi \otimes_{S[L]} S$ is the toric
  ideal $I_{L}$. One verifies (for instance in \emph{Macaulay2}~\cite{M2})
  that the other annihilator is $I_{L} \cap \ideal{x_2, x_4}$. In particular,
  the homology module is supported on an ideal other than the toric ideal, but
  differs only on the irrelevant locus. \qedhere
\end{example}

By introducing a small perturbation, \cite{And24}*{\S3} constructs a larger
family of $S$-complexes. From our perspective, we recover this construction by
making a small translation to the space $\RL$ in the ambient space $\RR^n$.
The next example, which notably lies outside the realm of ceiling
stratifications, illustrates this idea by adding a small parameter to
\Cref{e:And}.

\begin{example}[Cellular free resolution arising from a small perturbation]
  \label{e:And-ep}
  As in \Cref{e:diag} and \Cref{e:And}, consider the diagonal embedding of the
  second Hirzebruch surface.  Fix $0 < \varepsilon \ll 1$, and let
  $(\RL)_{\varepsilon} \colequal \RL + (0, 0, \varepsilon, 0) \subset \RR^n$
  have the standard cell structure, as shown in Figure~\ref{f:def}.  For the
  Anderson stratification $\psi \colon (\RL)_{\varepsilon} \to \ZZ^n$, the
  (nonminimal) $\Pic(X)$-graded cellular free $S$-complex
  $F_\psi \otimes_{S[L]} S$ is
  \[
    \begin{matrix}
      \; S(0,0, \; 0, 0)^2 \\[-3pt] 
      \oplus \\[-3pt]
      S(-1, 1, \; 1, -1) \\[-3pt] 
      \oplus \\[-3pt]
      S(-2, 1, \; 2, -1) \\ 
    \end{matrix}
    \xleftarrow{\qquad \partial_1 \qquad}
    \begin{matrix}
      \!\! S(0, 0, \; 0, 0) \\[-3pt]  
      \oplus \\[-3pt]
      \!\! S(-1, 0, \; -1, 0) \\[-3pt] 
      \oplus \\[-3pt]
      S(-1, 0, \; 0, -1)^2 \\[-3pt] 
      \oplus \\[-3pt]
      S(-2, 1, \; 1, -1)^2 \\[-3pt] 
      \oplus \\[-3pt]
      S(0, 0, \; 2, -1)^2 \\  
    \end{matrix}
    \xleftarrow{\qquad \partial_2 \qquad} \!\!\!\!\!\!
    \begin{matrix}
      \;\;\; S(-1, 0, \; -1, -1) \\[-3pt] 
      \oplus \\[-3pt]
      S(-2, 0, \; 0, -1) \\[-3pt] 
      \oplus \\[-3pt]
      S(-1, 0, \; 1, -1) \\[-3pt] 
      \oplus \\[-3pt]
      \;\; S(0, 0, \; 2, -1) \\ 
    \end{matrix}
    \!\!\!\! \xleftarrow{\qquad} 0,
  \]
  where%
  \vspace{-1.0em}
  \begin{align*}
    \partial_1
    &= \setlength{\arraycolsep}{2pt}
      \begin{bmatrix*}
         -1 & y_1 x_3 &  0      & -x_3 y_4 & 0   & 0   &  0 & y_2 \\
          1 & -x_1 y_3 & -x_1 y_4 & 0       & 0   & 0   &  y_2   & 0 \\
          0 &  0      &  y_1 x_4 & y_3 x_4 & -x_1 & -x_3 &  0  & 0 \\
          0 &  0      &  0       & 0      &  y_1 & y_3  & -x_2 & -x_2 \\
      \end{bmatrix*}
    && \!\!\!\! \text{and} \!\!\!\!
    & \partial_2
    &= \setlength{\arraycolsep}{2pt}
      \begin{bmatrix*}
        0   &  x_1 x_3 y_4 &  0       & -y_2 \\
        y_4 &  0           & -y_2     &  0 \\
        -y_3 &  x_3         &  0       &  0 \\
        y_1 & -x_1        &  0       &  0 \\
        0   & -y_3 x_4     &  x_2 x_3 &  0 \\
        0   &  y_1 x_4     & -x_1 x_2 &  0 \\
        0   &  0          & -x_1 y_3  &  1 \\
        0   &  0          &  y_1 x_3  &  -1 \\
      \end{bmatrix*}
    \, .
    \vspace{-1.0em}
  \end{align*}
  Combining \Cref{p:res} and \Cref{c:res} establishes that $F_\psi
  \otimes_{S[L]} S$ is a $\Pic(X)$-graded resolution of the $S$-module
  $M_{\psi} \otimes_{S[L]} S$.

  \begin{figure}[t]
    \centering

\begin{tikzpicture}[scale=1.5, line width=2pt]
\usetikzlibrary{calc, decorations.pathreplacing,shapes.misc, arrows.meta}
\usetikzlibrary{decorations.pathmorphing}

\coordinate (epsilony) at (0,-0.75);
\coordinate (epsilonx) at (1.5,0);
\coordinate  (origin) at (0,0);
\coordinate (origin2) at ($(origin)+(epsilonx)$);

\coordinate (br) at (4,0);
\coordinate (tr) at (4,4);
\coordinate (tr1) at ($(tr)+(epsilony)$);
\coordinate (tl) at (0,4);
\coordinate (tl1) at ($(tl)+(epsilony)$);
\coordinate (tl2) at ($(tl)+(epsilonx)$);
\coordinate (ml) at ($(0,2)+(epsilony)$);
\coordinate (mr) at ($(4,2)+(epsilony)$);

\coordinate (hoff) at (0.7,0);
\draw[RedViolet] ($(origin)-(hoff)$) -- ($(br)+(hoff)$);
\draw[RedViolet] ($(tl)-(hoff)$)--($(tr)+(hoff)$);

\coordinate (voff) at (0,0.5);
\draw[Cerulean] ($(origin)-(voff)$) -- ($(tl)+(voff)$);
\draw[Cerulean] ($(br)-(voff)$) -- ($(tr)+(voff)$);

\coordinate (diagoff) at (0.7,0.35);
\draw[LimeGreen] ($(origin2)-(diagoff)$) -- ($(mr)+(diagoff)$);
\draw[LimeGreen] ($(ml)-(diagoff)$) -- ($(tr1)+(diagoff)$);
\draw[LimeGreen] ($(tl1)-(diagoff)$) -- ($(tl2)+(diagoff)$);

\node[inner sep = 1.5pt, fill=white] at (origin)
{\footnotesize $1$};
\node[inner sep = 1.5pt, fill=white] (ml-label) at (ml)
{\footnotesize $\frac{x_3 y_4}{y_3 x_4}$};
\node[inner sep = 1.5pt, fill=white] at (tl)
{\footnotesize $\frac{x_2^{} x_3^2 y_4^{}}{y_2^{} y_3^2 x_4^{}}$};
\node[inner sep = 1.5pt, fill=white] at (br)
{\footnotesize $\frac{x_1y_3}{y_1 x_3}$};
\node[inner sep = 1.5pt, fill=white] (mr-label) at (mr)
{\footnotesize $\frac{x_1 y_4}{y_1 x_4}$};
\node[inner sep = 1.5pt, fill=white] at (tr)
{\footnotesize $\frac{x_1 x_2 x_3 y_4}{y_1 y_2 y_3 x_4}$};
\node[inner sep = 1.5pt, fill=white] at (origin2)
{\footnotesize $1$};
\node[inner sep = 1.5pt, fill=white] at (tl1)
{\footnotesize $\frac{x_3^2 y_4^{}}{y_3^2 x_4^{}}$};
\node[inner sep = 1.5pt, fill=white] at (tl2)
{\footnotesize $\frac{x_2^{} x_3^2 y_4^{}}{y_2^{} y_3^2 x_4^{}}$};
\node[inner sep = 1.5pt, fill=white] (tr1-label) at (tr1)
{\footnotesize $\frac{x_1 x_3 y_4}{y_1 y_3 x_4}$};


\end{tikzpicture} 
    \vspace{-0.75em}
    \caption{A fundamental domain of an $\varepsilon$-shifted stratification
      for the diagonal embedding of the second Hirzebruch surface.}
    \label{f:def}
    \vspace*{-0.5em}
  \end{figure}
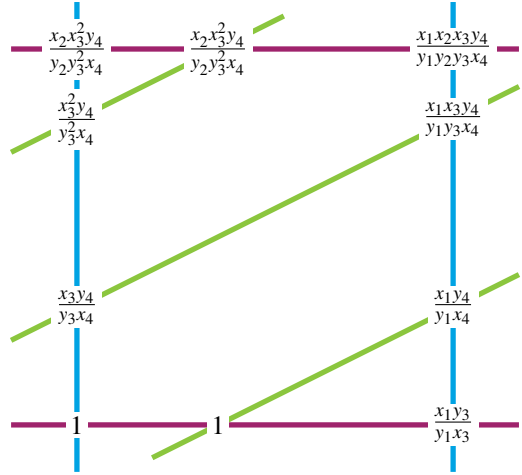

  We also claim that the $\mathcal{O}_{X}$-complex associated to
  $F_{\psi} \otimes_{S[L]} S$ is a resolution of the diagonal.  The monomial
  $S$-module $M_{\psi} \otimes_{S[L]}S$ is generated by the Laurent monomials
  $1$, $\smash{x_{3}^{} y_{3}^{-1} x_{4}^{-1} y_{4}^{}}$ and
  $\smash{x_{3}^{2} y_{3}^{-2} x_{4}^{-1} y_{4}^{}}$.  Applying \Cref{p:sheaf}
  and \Cref{r:irr}, we check that
  $\smash{(\bm{x}^{\widehat{\sigma}})^{k} \!\cdot\! M_{\psi} \subseteq M_{L}}$
  for the relevant cones $\sigma \in \Sigma_{X}$ as follows:
  \begin{align*}
    x_1^{} y_1^{} x_2^{} y_2^{} \!\cdot\!
      \tfrac{x_3^{} y_4^{}}{y_3^{} x_4^{}}
    &= y_1^2 y_2^2 \!\cdot\!
      \tfrac{x_1^{} x_2^{} x_3^{} y_4^{}}{y_1^{} y_2^{} y_3^{} x_4^{}}, \!
    & x_2^{} y_2^{} x_3^{} y_3^{} \!\cdot\!
      \tfrac{x_3^{} y_4^{}}{y_3^{} x_4^{}}
    &= y_2^2 y_3^2 \!\cdot\!
      \tfrac{x_2^{} x_3^2 y_4^{}}{y_2^{} y_3^{2} x_4^{}}, \!
    & x_1^{} y_1^{} x_4^{} y_4^{} \!\cdot\!
      \tfrac{x_3^{} y_4^{}}{y_3^{} x_4^{}}
    &= x_1^2 y_4^2 \!\cdot\!
      \tfrac{y_1^{} x_3^{}}{x_1^{} y_3^{}}, \! \\[-2pt]
    x_1^{} y_1^{} x_2^{} y_2^{} \!\cdot\!
      \tfrac{x_3^2 y_4^{}}{y_3^{2} x_4^{}}
    &= x_1^{} y_1^{} y_2^{2} \!\cdot\!
      \tfrac{x_2^{} x_3^{2} y_4^{}}{y_2^{} y_3^{2} x_4^{}}, \!
    & x_2^{} y_2^{} x_3^{} y_3^{} \!\cdot\!
      \tfrac{x_3^2 y_4^{}}{y_3^{2} x_4^{}}
    &= y_2^{2} x_3^{} y_3^{} \!\cdot\!
      \tfrac{x_2^{} x_3^2 y_4^{}}{y_2^{} y_3^{2} x_4^{}}, \!
    & (x_1^{} y_1^{} x_4^{} y_4^{})^2 \!\cdot\!
      \tfrac{x_3^2 y_4^{}}{y_3^{2} x_4^{}}
    &= x_1^{4} x_4^{} y_4^{3} \!\cdot\!
      \left( \! \tfrac{y_1^{} x_3^{}}{x_1^{} y_3^{}} \! \right)^{\!\!2} . \!
  \end{align*}
  Hence, the $\mathcal{O}_{X}$-module associated to $M_{\psi} \otimes_{S[L]}S$
  is $\Delta_{*}^{} \mathcal{O}_{Y}$.
\end{example}

Although these small perturbations often produce $S$-complexes that come
closer to satisfying \Cref{p:res}, it is not clear whether there is always a
choice of $\varepsilon$ that yields a free resolution.

\subsection*{Cellular interpretation of the Favero--Huang homotopy path
  algebra construction}

We also discuss how the resolution of the diagonal in \cite{FH25}*{\S5} fits
into our framework.  To sidestep the technical details surrounding
\cite{FH25}*{Corollary~6.7}, which describes a projective cellular resolution
of the diagonal bimodule over a homotopy path algebra, we concentrate on the
toric variety $Y = \PP^{2}$.

Let $\Delta \colon \PP^{2} \to \PP^2 \times \PP^2$ denote the diagonal
embedding, and let $S \colequal \kk[x_1, x_2, x_3, y_1, y_2, y_3]$ be the Cox
ring of $\PP^2 \times \PP^2$.  The corresponding lattice $L \subseteq \RR^{6}$
is generated by the integral vectors $(1, -1, 0, -1, 1, 0)$ and
$(0, -1, 1, 0, 1, -1)$.  These vectors form a basis of $L$, which determines
an isomorphism $\iota \colon \ZZ^{2} \to L \subset \ZZ^{6}$.  Extending the
map $\iota$ linearly gives the map $\iota_{\RR} \colon \RR^{2} \to \RR^{6}$
whose image is exactly $\RL$.

To describe the desired cell structure on the space $\RL$, consider the closed
subset
\[
  \mathcal{S}
  = \overline{ \set{\bm{p} \in \RR^3 \mathrel{\big|}
      \text{$\ceil{p_1} + \ceil{p_2} + \ceil{p_3} = 0$
        and at least one coordinate $p_i \in \ZZ$}}}
\]
in $\RR^3$ (with its Euclidean topology). \cite{MS05}*{\S3} calls
$\mathcal{S}$ a ``staircase surface.'' We give $\mathcal{S}$ the polyhedral
cell structure determined by its intersection with the periodic arrangement
arising from the standard basis in $\RR^{3}$; compare with \Cref{d:std}. The
maximal cells in $\mathcal{S}$ are unit squares.  The linear projection of
$\mathcal{S}$ onto the linear subspace $V \subset \RR^{3}$ along the vector
$(1,1,1)$ endows the space~$V$ with a polyhedral cell structure. Since $V$ is
spanned by $(1,-1,0)$ and $(0,-1,1)$, the Lawrence lift of $V$ is exactly
$\RL$, so $V$ is homeomorphic to $\RL$ and $\RL$ acquires a polyhedral cell
structure.

The relevant compatible $\ZZ^{6}$-stratification depends on an auxiliary
function $\gamma \colon \RR^{2} \to \mathcal{S}$.  This map is defined in two
steps.  First, include $\RR^{2}$ into $\RR^{3}$ via the linear map determined
by the matrix
\[
  \begin{bmatrix*}
    1 & 0 \\
    -1 & -1 \\
    0 & 1 \\
  \end{bmatrix*} \, .
\]
Next, project the point in $\RR^{3}$ onto $\mathcal{S}$ along the direction
$(1,1,1)$.  This composite map $\gamma$ is a homeomorphism. Moreover, the
image of the standard integral lattice $\ZZ^{2} \subset \RR^{2}$ under
$\gamma$ is $\set{ \bm{u} \in \mathcal{S} \mathrel{\big|} u_1 + u_2 + u_3 = 0}
\cong \ZZ^{2}$. For any lattice point $\bm{u} \in \ZZ^{2}$, observe that
$\iota(\bm{u}) = \bigl( \gamma(\bm{u}), - \gamma(\bm{u}) \kern-1.0pt \bigr)$.

\begin{lemma}
  \label{l:FH}
  The map $\psi \colon \RL \to \ZZ^{6}$, defined by
  $\psi(\iota_\RR(\bm{q})) = \ceil{ \bigl( \gamma(\bm{q}), - \gamma(\bm{q})
    \bigr)}$ for all $\bm{q} \in \RR^{2} \cong \RL$, is a compatible
  $\ZZ^6$-stratification.
\end{lemma}

\Cref{f:stair} displays part of the cell complex $\RL$ in which each vertex
$\bm{q} \in \RL$ is labeled with the monomial having exponent vector
$\psi(\bm{q})$.  This agrees with the cell structure in
\cite{FH25}*{Example~3.15}.

\begin{proof}[Proof of \Cref{l:FH}]
  By construction, the map $\psi$ is constant on the open cells in $\RL$.
  Moreover, for any point $\bm{q} \in \RR^{2}$ and any lattice point
  $\bm{u} \in \ZZ^{2}$, we have
  $\gamma(\bm{q} + \bm{u}) = \gamma(\bm{q}) + \gamma(\bm{u})$.  It follows
  that
  \begin{align*}
    \psi \bigl( \iota_{\RR}(\bm{q} + \bm{u}) \kern-1.0pt \bigr)
    = \ceil{\bigl( \gamma(\bm{q} + \bm{u}),
      - \gamma(\bm{q} + \bm{u}) \kern-1.0pt \bigr)}
    &= \ceil{\bigl( \gamma(\bm{q}), -\gamma(\bm{q}) \kern-1.0pt \bigr)}
      + \ceil{\bigl( \gamma(\bm{u}), -\gamma(\bm{u}) \kern-1.0pt \bigr)} \\
    &= \psi \bigl( \iota_{\RR}(\bm{q}) \kern-1.0pt \bigr)
      + \bigl( \gamma(\bm{u}), -\gamma(\bm{u}) \kern-1.0pt \bigr)
    = \psi \bigl( \iota_{\RR}(\bm{q}) \kern-1.0pt \bigr)
      + \iota(\bm{u}) \, ,
  \end{align*}
  which shows that $\psi$ is compatible with lattice translation.

  To demonstrate that the map $\psi$ is continuous, we must show that, for a
  fixed lattice point $\bm{u} \in \ZZ^{6}$, the subset
  $(\RL)_{\leqslant \bm{u}} = \set{\bm{p} \in \RL \mathrel{\big|} \psi(\bm{p})
    \leqslant \bm{u}}$ is closed in $\RL$.  Lifting the defining inequality to
  the ambient space $\RR^{3}$, we see that $(\RL)_{\leqslant \bm{u}}$ is the
  projection from $\RR^{3}$ of the intersection of the set $\mathcal{S}$ with
  the closed rectangular cuboid
  $[-u_4, u_1] \times [-u_5, u_2] \times [-u_6, u_3]$.
\end{proof}

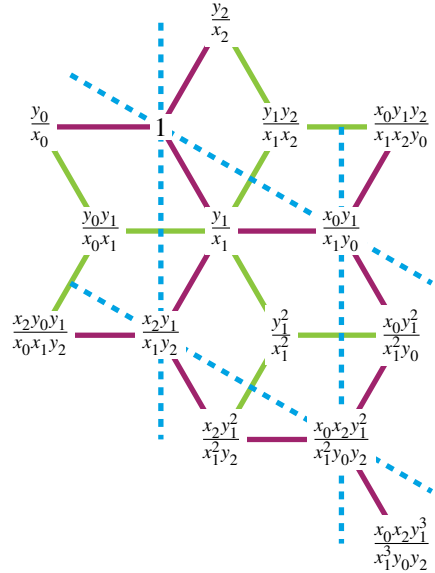
\begin{figure}[ht]
  \centering

\begin{tikzpicture}[x={(1cm,0cm)},y={(-0.5cm,0.866cm)},z={(-0.5cm,-0.866cm)},
  scale=1.6, line width=2pt]
  \usetikzlibrary{calc, decorations.pathreplacing,shapes.misc, arrows.meta}
  \usetikzlibrary{decorations.pathmorphing}

  \coordinate (origin) at (0,0,0);
  \coordinate (a) at (1,-1,0);
  \coordinate (b) at (0,-1,1);
  \coordinate (ab) at ($(a)+(b)$);

  \draw[RedViolet] (origin) -- (-1,0,0);
  \draw[RedViolet] (origin) -- (0,-1,0);
  \draw[RedViolet] (origin) -- (0,0,-1);

  \draw[RedViolet] (a) -- ++(-1,0,0);
  \draw[RedViolet] (a) -- ++(0,-1,0);
  \draw[RedViolet] (a) -- ++(0,0,-1);

  \draw[RedViolet] (b) -- ++(-1,0,0);
  \draw[RedViolet] (b) -- ++(0,-1,0);
  \draw[RedViolet] (b) -- ++(0,0,-1);

  \draw[RedViolet] (ab) -- ++(-1,0,0);
  \draw[RedViolet] (ab) -- ++(0,-1,0);
  \draw[RedViolet] (ab) -- ++(0,0,-1);

  \draw[LimeGreen] (0,-1,-1) -- ++(1,0,0);
  \draw[LimeGreen] (0,-1,-1) -- ++(0,1,0);
  \draw[LimeGreen] (0,-1,-1) -- ++(0,0,1);

  \draw[LimeGreen] (-1,-1,0) -- ++(1,0,0);
  \draw[LimeGreen] (-1,-1,0) -- ++(0,1,0);
  \draw[LimeGreen] (-1,-1,0) -- ++(0,0,1);

  \draw[LimeGreen] (0,-2,0) -- ++(1,0,0);
  \draw[LimeGreen] (0,-2,0) -- ++(0,1,0);
  \draw[LimeGreen] (0,-2,0) -- ++(0,0,1);

  \draw[dashed,Cerulean] ($(origin)-0.5*(a)$) -- ($1.5*(a)$);
  \draw[dashed,Cerulean] ($(origin)-0.5*(b)$) -- ($1.5*(b)$);
  \draw[dashed,Cerulean] ($(a)-0.5*(b)$) -- ($(a)+1.5*(b)$);
  \draw[dashed,Cerulean] ($(b)-0.5*(a)$) -- ($1.5*(a)+(b)$);

  \node[inner sep = 1.5pt, fill=white] at (origin)
  {\footnotesize $1$};
  \node[inner sep = 1.5pt, fill=white] at (a)
  {\footnotesize $\frac{x_1^{} y_2^{}}{x_2^{} y_1^{}}$};
  \node[inner sep = 1.5pt, fill=white] at (b)
  {\footnotesize $\frac{x_3^{}y_2^{}}{x_2^{} y_3^{}}$};
  \node[inner sep = 1.5pt, fill=white] at ($(a)+(b)$)
  {\footnotesize $\frac{x_1^{} x_3^{} y_2^2}{x_2^2 y_1^{} y_3^{}}$};

  \node[inner sep = 1.5pt, fill=white] at (-1,0,0)
  {\footnotesize $\frac{y_1^{}}{x_1^{}}$};
  \node[inner sep = 1.5pt, fill=white] at (0,-1,0)
  {\footnotesize $\frac{y_2^{}}{x_2^{}}$};
  \node[inner sep = 1.5pt, fill=white] at (0,0,-1)
  {\footnotesize $\frac{y_3^{}}{x_3^{}}$};

  \node[inner sep = 1.5pt, fill=white] at ($(a)+(0,-1,0)$)
  {\footnotesize $\frac{x_1^{} y_2^2}{x_2^2 y_1^{}}$};
  \node[inner sep = 1.5pt, fill=white] at ($(a)+(0,0,-1)$)
  {\footnotesize $\frac{x_1^{} y_2^{} y_3^{}}{x_2^{} x_3^{} y_1^{}}$};

  \node[inner sep = 1.5pt, fill=white] at ($(b)+(-1,0,0)$)
  {\footnotesize $\frac{x_3^{} y_1^{} y_2^{}}{x_1^{} x_2^{} y_3^{}}$};
  \node[inner sep = 1.5pt, fill=white] at ($(b)+(0,-1,0)$)
  {\footnotesize $\frac{x_3^{} y_2^2}{x_2^2 y_3^{}}$};

  \node[inner sep = 1.0pt, fill=white] at ($(ab)+(0,-1,0)$)
  {\footnotesize $\frac{x_1^{} x_3^{} y_2^3}{x_2^3 y_1^{} y_3^{}}$};

  \node[inner sep = 1.5pt, fill=white] at (0,-1,-1)
  {\footnotesize $\frac{y_2^{} y_3^{}}{x_2^{} x_3^{}}$};
  \node[inner sep = 1.5pt, fill=white] at (0,-2,0)
  {\footnotesize  $\frac{y_2^2}{x_2^2}$};
  \node[inner sep = 1.5pt, fill=white] at (-1,-1,0)
  {\footnotesize $\frac{y_1^{} y_2^{}}{x_1^{} x_2^{}}$};
\end{tikzpicture}
  \vspace{-0.75em}
  \caption{A portion of the labeled cell complex giving a Favero--Huang
    resolution of the diagonal for $\PP^2$.  The dashed lines illustrate the
    underlying lattice.}
  \label{f:stair}
  \vspace{-1.0em}
\end{figure}

Unlike the ceiling stratification, the subsets $(\RL)_{\leqslant \bm{u}}$ are
not always convex.  For example, the subset corresponding to the lattice point
$\bm{u} \colequal (0, 0, -1, 0, 0, -1)$ is a union of two quadrilaterals
meeting at a single vertex. However, the cellular free $S$-complex $F_\psi
\otimes_{S[L]} S$ is
still a resolution.

\begin{proposition}
  \label{p:FH}
  When $\psi \colon \RL \to \ZZ^{6}$ is the compatible
  $\ZZ^{6}$-stratification defined in \Cref{l:FH}, the cellular free
  $S$-complex $F_\psi \otimes_{S[L]} S$ is a resolution of the monomial module
  $M_{\psi} \otimes_{S[L]} S$.
\end{proposition}

\begin{proof}
  As $\mathcal{S}$ is a ``staircase surface'' aligned with the coordinate
  hyperplane, the intersection of $\mathcal{S}$ with a coordinate hyperplane
  aligned rectangular cuboid yields a subset of this staircase such that the
  projection onto a coordinate hyperplane yields a square.  This projection is
  a homotopy equivalence contracting the ``vertical'' planes of this
  staircase, consequently the intersection is contractible.  Hence, the subset
  $(\RL)_{\leqslant \bm{u}}$ is contractible for all $\bm{u} \in
  \RL$. Combining \Cref{p:res} and \Cref{c:res} shows that $F_\psi
  \otimes_{S[L]} S$ is a resolution.
\end{proof}

For the compatible $\ZZ^{6}$-stratification defined in \Cref{l:FH}, the
monomial module $\smash{M_{\psi} \otimes_{S[L]} S}$ is generated by the
Laurent monomials $1$, $\smash{x_1^{} y_{1}^{-1}}$, and
$\smash{x_{1}^{} y_{1}^{-1} x_{2}^{} y_{2}^{-1}}$. The (minimal)
$\ZZ^{2}$-graded cellular free $S$-complex $F_\psi \otimes_{S[L]} S$ is
\vspace*{-1.0em}
\[
  \begin{matrix}
    S(0,0) \\[-3pt]
    \oplus \\[-3pt]
    S(1,-1) \\[-3pt]
    \oplus \\[-3pt]
    S(2,-2) \\
  \end{matrix}
  \xleftarrow{
    \begin{bmatrix}
      0 & 0 & 0 & -y_1 & -y_2 & -y_3 \\
      -y_1 & -y_2 & -y_3 & x_1 & x_2 & x_3 \\
      x_1 & x_2 & x_3 & 0 & 0 & 0 \\
    \end{bmatrix}}
  \begin{matrix}
    S(1,-2)^3 \\[-3pt]
    \oplus \\[-3pt]
    S(0,-1)^3 \\
  \end{matrix}
  \xleftarrow{
    \begin{bmatrix}
      0 & -x_3 & -x_2 \\
      -x_3 & 0 & x_1 \\
      x_2 & x_1 & 0\\
      0 & y_3 & y_2 \\
      y_3 & 0 & -y_1 \\
      -y_2 & -y_1 & 0 \\
    \end{bmatrix}}
  \begin{matrix}
    S(0,-2)^3  \\
  \end{matrix}
  \xleftarrow{\quad}
  \begin{matrix}
    0 \\
  \end{matrix}
  \, .
\]
Generalizing this construction recovers the resolutions of the diagonal for
$\PP^{n}$ implicitly described in \cite{FH25}*{Corollary~6.7}. It is unclear
whether a similar approach can be used to construct resolutions of the
diagonal for other toric varieties.

\subsection*{Acknowledgements}

This project began at the American Institute of Mathematics (AIM) in 2023 and
continued at the Simon Laufer Mathematical Institute (SLMath) in
2024. Computational experiments in \emph{Macaulay2}~\cite{M2} were
essential. We are grateful to Reginald Anderson, Isidora Bailly-Hall, Michael
Brown, David Eisenbud, Daniel Erman, David Favero, Andrew Hanlon, and Mahrud
Sayrafi for helpful discussions. We also thank the anonymous referee for helpful feedback.
CB was partially supported by the NSF Grants
DMS-2001101 and DMS-2412039; LCH was partially supported by the NSF Grants MSPRF
DMS-2503497, DMS-2401482, and DMS-1901848; and GGS was partially supported by
NSERC.

\subsection*{AI Declaration}

No AI tools were used in the work or preparation of this article.

\raggedright

\raggedright

\end{document}